\documentclass[10pt]{article}

\usepackage{amssymb}
\usepackage{amsfonts}
\usepackage{amsthm}
\usepackage{amsmath}
\usepackage{float}
\usepackage{bigints}
\usepackage{fullpage}

\newtheorem{Lemma}{Lemma}
\newtheorem{Theorem}{Theorem}

\newtheorem{Problem}{Problem}
\newtheorem{Corollary}{Corollary}
\newtheorem{Remark}{Remark}

\numberwithin{Subcase}{Case}

\usepackage[pdftex]{graphicx}
\graphicspath{{C:/Users/Sam/Desktop/PaperFigures/}}
\DeclareGraphicsExtensions{.pdf}

\title{Contact Graphs of Unit Sphere Packings Revisited
\footnote{Keywords: congruent sphere packing, touching pairs, triplets, quadruples, density, (truncated) Voronoi cell, union of balls, isoperimetric inequality, spherical cap packing.  
2010 Mathematics Subject Classification: 52C17, 05B40, 11H31, and 52C45.}}

\author{K\'{a}roly Bezdek\thanks{Partially supported by a Natural Sciences and 
Engineering Research Council of Canada Discovery Grant.} and Samuel Reid\thanks{Supported by an Undergraduate Student Research Award of the Natural Sciences and Engineering Research Council of Canada.}}

\begin{document}

\maketitle

\begin{abstract}
The contact graph of an arbitrary finite packing of unit balls in Euclidean 3-space is the (simple) graph whose vertices correspond to the packing elements and whose two vertices are connected by an edge if the corresponding two packing elements touch each other. One of the most basic questions on contact graphs is to find the maximum number of edges that a contact graph of a packing of n unit balls can have. In this paper, improving earlier estimates, we prove that the number of touching pairs in an arbitrary packing of $n$ unit balls in $\mathbb{E}^{3}$ is always less than $6n-0.926n^{\frac{2}{3}}$. Moreover, as a natural extension of the above problem, we propose to study the maximum number of touching triplets (resp., quadruples) in an arbitrary packing of $n$ unit balls in Euclidean 3-space. In particular, we prove that the number of touching triplets (resp., quadruples) in an arbitrary packing of $n$ unit balls in $\mathbb{E}^3$ is at most $\frac{25}{3}n$ (resp., $\frac{11}{4}n$).
\end{abstract}

\section{Introduction}
Let $\mathbb{E}^{d}$ denote $d$-dimensional Euclidean space. Then the {\it contact graph} of an arbitrary finite packing of unit balls (i.e., of an arbitrary finite family of non-overlapping balls having unit radii) in $\mathbb{E}^{d}$ is the (simple) graph whose vertices correspond to the packing elements and whose two vertices are connected by an edge if and only if the corresponding two packing elements touch each other. One of the most basic questions on contact graphs is to find the maximum number of edges that a contact graph of a packing of $n$ unit balls can have in $\mathbb{E}^{d}$. In 1974 Harborth \cite{Ha} proved the following optimal result in $\mathbb{E}^{2}$: the maximum number of touching pairs in a packing of $n$ congruent circular disks in $\mathbb{E}^{2}$ is precisely $\lfloor 3n-\sqrt{12n-3}\rfloor$. In dimensions three and higher only estimates are known for the maximum number of touching pairs. In particular, just very recently the first named author \cite{B12} proved that the number of touching pairs in an arbitrary packing of $n$ unit balls in $\mathbb{E}^{3}$ is always less than $6n-0.695n^{\frac{2}{3}}$. Moreover, it is proved in \cite{B02} that for $d\ge4$ the number of touching pairs in an arbitrary packing of $n$ unit balls in $\mathbb{E}^{d}$ is less than
$$\frac{1}{2}\tau_d\, n-\frac{1}{2^{d}}\delta_d^{-\frac{d-1}{d}}\; n^{\frac{d-1}{d}},$$
where $\tau_d$ stands for the kissing number of a unit ball in $\mathbb{E}^{d}$ (i.e., it denotes the maximum number of non-overlapping unit balls of $\mathbb{E}^{d}$ that can touch a given unit ball in $\mathbb{E}^{d}$) and $\delta_d$ denotes the largest possible density for (infinite) packings of unit balls in $\mathbb{E}^{d}$. For a nice survey on recognition-complexity results of ball contact graphs we refer the interested reader to \cite{HK01}.

In this paper, first we improve the above quoted upper bound of \cite{B12} as follows.

\bigskip

\begin{Theorem}\label{improved-estimate}
\item (i)
The number of touching pairs in an arbitrary packing of $n\ge 2$ unit balls in $\mathbb{E}^{3}$ is always less than $6n-0.926n^{\frac{2}{3}}$.
\item (ii)
The number of touching pairs in an arbitrary lattice packing of $n\ge 2$ unit balls in $\mathbb{E}^{3}$ is always less than $ 6n-\frac{3\sqrt[3]{18\pi}}{\pi}n^{\frac{2}{3}}=6n-3.665\dots n^{\frac{2}{3}}$.

\end{Theorem}

In connection with Theorem~\ref{improved-estimate} we recall from \cite{B12} that for all $n=\frac{2k^{3} +k}{3}, k \geq 2$, there are packings of $n$ unit balls in $\mathbb{E}^3$ such that the number of touching pairs is greater than $6n-\sqrt[3]{486}n^{\frac{2}{3}}=6n-7.862\dots n^{\frac{2}{3}}$.

Second,  as a natural extension of the above discussed problem on the number of touching pairs, we propose to study the maximum number of touching triplets, quadruples, etc. in an arbitrary packing of $n$ unit balls in $\mathbb{E}^{d}$. Harborth's proof \cite{Ha} implies in a straightforward way that the maximum number of touching triplets in a packing of $n$ congruent circular disks in $\mathbb{E}^{2}$ is precisely $\lfloor 3n-\sqrt{12n-3}\rfloor-n+1$. In this paper we study the $3$-dimensional case of the problem at hand and prove the following estimates.

\begin{Theorem}\label{main-theorem}
\item (i)
The number of touching triplets (resp., quadruples) in an arbitrary packing of $n\ge 3$ (resp., $n\ge 4$) unit balls in $\mathbb{E}^3$ is at most $\frac{25}{3}n$ (resp., $\frac{11}{4}n$).
\item(ii)
The number of touching  triplets (resp., quadruples) in an arbitrary lattice packing of $n\ge 2$ unit balls in $\mathbb{E}^{3}$ is at most $8n$ (resp., $2n$).
\item (iii)
For all $n=\frac{2k^{3} +k}{3}, k \geq 2$, there are packings of $n$ unit balls (with their centers lying on a face-centered cubic lattice) in $\mathbb{E}^3$ such that the number of touching triplets (resp., quadruples) is 
$$\frac{4}{3}(k-1)k(4k-5)>8n - 12\left(\frac{3}{2}n\right)^{2/3}+4n^{1/3}\ \ \ \ ({\rm resp}.,\ \frac{4}{3}(k-2)(k-1)k>2n-4(\frac{3}{2}n)^{2/3} + 2n^{1/3})\ .$$
\end{Theorem}

Part $(iii)$ of Theorem~\ref{main-theorem} raises the following question.

\begin{Problem}
Prove or disprove the existence of a positive integer $k_3$ (resp., $k_4$) with the property that for any positive integer $k\ge k_3$ (rep., $k\ge k_4$) the number of touching triplets (resp., quadruples) in an arbitrary packing of $\frac{2k^3+k}{3}$ unit balls in $\mathbb{E}^3$ is at most $\frac{4}{3}(k-1)k(4k-5)$ (resp., $\frac{4}{3}(k-2)(k-1)k $).
\end{Problem}

The proof of $(i)$ in Theorem~\ref{main-theorem} is based on the following statement that might be of independent interest in particular, because one can regard that statement as a spherical analogue of Harborth's theorem \cite{Ha} for the angular radius $\pi/6$ (on the unit sphere $\mathbb{S}^2$ centered at the origin in $\mathbb{E}^3$).

\begin{Theorem}\label{spherical-main-theorem}
The number of touching pairs (resp., triplets) in an arbitrary packing of spherical caps of angular radius $\pi/6$ on $\mathbb{S}^2$ is at most $25$ (resp., $11$).
\end{Theorem}

In connection with Theorem~\ref{spherical-main-theorem} we put forward the following question.

\begin{Problem}\label{main-conjecture}
Prove or disprove that the number of touching pairs (resp., triplets) in an arbitrary packing of spherical caps of angular radius $\pi/6$ on $\mathbb{S}^2$ is at most $24$ (resp., $10$).
\end{Problem}

We note that a positive answer to Problem~\ref{main-conjecture} implies the following improvement on the estimates in $(i)$ of Theorem~\ref{main-theorem}: the number of touching triplets (resp., quadruples) in an arbitrary packing of $n$ unit balls in $\mathbb{E}^3$ is at most $8n$ (resp., $\frac{5}{2}n$).

Due to the Minkowski difference body method (see for example, Chapter 6 in \cite{Ro64}) the family ${\cal P}_{\mathbf{K}}:=\{\mathbf{t}_1+\mathbf{K}, \mathbf{t}_2+\mathbf{K}, \dots , \mathbf{t}_n+\mathbf{K}\} $ of $n$ translates of the convex body $\mathbf{K}$ in $\mathbb{E}^3$ is a packing if and only if the family ${\cal P}_{\mathbf{K}_{\mathbf{o}}}:=\{\mathbf{t}_1+\mathbf{K}_{\mathbf{o}}, \mathbf{t}_2+\mathbf{K}_{\mathbf{o}}, \dots , \mathbf{t}_n+\mathbf{K}_{\mathbf{o}}\} $ of $n$ translates of the symmetric difference body $\mathbf{K}_{\mathbf{o}}:=\frac{1}{2}(\mathbf{K}+(-\mathbf{K}))$ of $\mathbf{K}$ is a packing in $\mathbb{E}^3$. Moreover, the number of touching pairs, triplets, and quadruples in the packing ${\cal P}_{\mathbf{K}}$ is equal to the number of touching pairs, triplets, and quadruples in the packing ${\cal P}_{\mathbf{K}_{\mathbf{o}}}$. Thus, for this reason and for the reason that if $\mathbf{K}$ is a convex body of constant width in $\mathbb{E}^3$, then $\mathbf{K}_{\mathbf{o}}$ is a ball of $\mathbb{E}^3$, Theorem~\ref{improved-estimate} as well as Theorem~\ref{main-theorem} extend in a straightforward way to translative packings of convex bodies of constant width in $\mathbb{E}^3$.

The rest of the paper is organized as follows. In Section 2 we prove $(i)$ of  Theorem~\ref{improved-estimate}. Section 3 proves $(ii)$ of Theorem~\ref{improved-estimate} as well as $(ii)$ of 
Theorem~\ref{main-theorem}.  In Section 4 we give a short proof of $(i)$ in Theorem~\ref{main-theorem} using Theorem~\ref{spherical-main-theorem}. Sections 5 and 6 present our elementary and somewhat computational proof of Theorem~\ref{spherical-main-theorem}. Finally, Section 7 gives a proof of $(iii)$ in Theorem~\ref{main-theorem} as well as shows that if Problem~\ref{main-conjecture} has a positive answer, then its estimates are tight.

\section{Proof of (i) in Theorem~\ref{improved-estimate}}

The proof presented in this section follows the ideas of the proof of $(i)$ of Theorem 1.1 in \cite{B12} with some proper modifications based on the recent breakthrough results of Hales \cite{H11}.
The details are as follows.

Let $\mathbf{B}$ denote the (closed) unit ball centered at the origin $\mathbf{o}$ of $\mathbb{E}^3$ and let ${\cal P}:=\{\mathbf{c}_1+\mathbf{B}, \mathbf{c}_2+\mathbf{B}, \dots , \mathbf{c}_n+\mathbf{B}\} $ denote the packing of $n$ unit balls with centers $\mathbf{c}_1, \mathbf{c}_2, \dots , \mathbf{c}_n$ in $\mathbb{E}^3$ having the largest number $C(n)$ of touching pairs among all packings of $n$ unit balls in $\mathbb{E}^3$. (${\cal P}$ might not be uniquely determined up to congruence in which case ${\cal P}$ stands for any of those extremal packings.) Now, let $\hat{r}:=1.58731$. The following statement shows the main property of $\hat{r}$ that is needed for our proof of Theorem~\ref{improved-estimate}.

\begin{Theorem}\label{Bezdek-Hales-1}
Let $\mathbf{B}_1, \mathbf{B}_2, \dots , \mathbf{B}_{13}$ be $13$ different members of a packing of unit balls in $\mathbb{E}^3$. Assume that each ball of the family $\mathbf{B}_2, \mathbf{B}_3, \dots , \mathbf{B}_{13}$ touches $\mathbf{B}_1$. Let $\hat{\mathbf{B}}_i$ be the closed ball concentric with $\mathbf{B}_i$ having radius $\hat{r}$, $1\le i\le 13$. Then
the boundary ${\rm bd}(\hat{\mathbf{B}}_1)$ of $\hat{\mathbf{B}}_1$ is covered by the balls $\hat{\mathbf{B}}_2, \hat{\mathbf{B}}_3, \dots , \hat{\mathbf{B}}_{13}$, that is,
$${\rm bd}(\hat{\mathbf{B}}_1)\subset \cup_{j=2}^{13}\hat{\mathbf{B}}_j \ .$$
\end{Theorem}

\proof
Let $\mathbf{o}_i$ be the center of the unit ball $\mathbf{B}_i$, $1\le i\le 13$ and assume that 
 $\mathbf{B}_1$ is tangent to the unit balls $\mathbf{B}_2, \mathbf{B}_3, \dots , \mathbf{B}_{13}$ at the points $\mathbf{t}_j\in{\rm bd}(\mathbf{B}_j)\cap{\rm bd}(\mathbf{B}_1), 2\le j\le 13$.
 
Let $\alpha$ denote the measure of the angles opposite to the equal sides of the isosceles triangle $\triangle\mathbf{o}_1\mathbf{p}\mathbf{q}$ with ${\rm dist}(\mathbf{o}_1,\mathbf{p})=2$ and ${\rm dist}(\mathbf{p},\mathbf{q})={\rm dist}(\mathbf{o}_1,\mathbf{q})=\hat{r}$, where ${\rm dist}(\cdot , \cdot )$ denotes the Euclidean distance between the corresponding two points. Clearly, 
$\cos\alpha =\frac{1}{\hat{r}}$ with $\alpha<\frac{\pi}{3}$. 

\begin{figure}[h!]
\begin{center}
\includegraphics[scale=0.25]{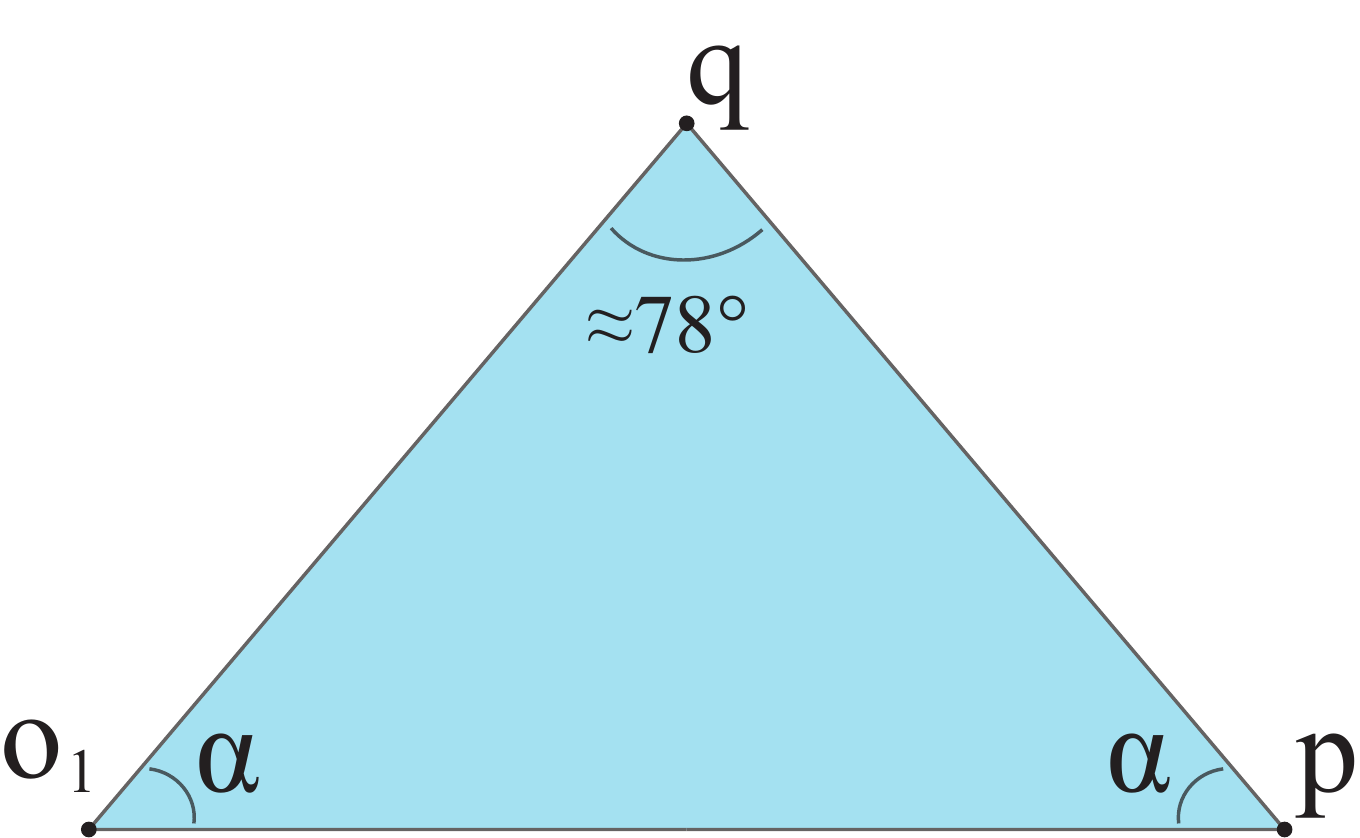}
\caption{The isosceles triangle $\triangle\mathbf{o}_1\mathbf{p}\mathbf{q}$.}
\label{fig:AlphaTriangle}
\end{center}
\end{figure}

\begin{Lemma}\label{B-B-S}
Let $\mathbf{T}$ be the convex hull of the points $\mathbf{t}_2, \mathbf{t}_3, \dots , $ $\mathbf{t}_{13}$. Then the radius of the circumscribed circle of each face of the convex polyhedron $\mathbf{T}$ is less than $\sin\alpha$.
\end{Lemma} 

\proof
Let $F$ be an arbitrary face of $\mathbf{T}$ with vertices $\mathbf{t}_j, j\in I_F\subset\{2, 3, \dots ,$ $ 13\}$ and let $\mathbf{c}_F$ denote the center of the circumscribed circle of $F$. Clearly, the triangle $\triangle\mathbf{o}_1\mathbf{c}_F\mathbf{t}_j$ is a right triangle with a right angle at $\mathbf{c}_F$ and with an acute angle of measure $\beta_F$ at $\mathbf{o}_1$ for all $j\in I_F$.
We have to show that $\beta_F<\alpha$. We prove this by contradiction. Namely, assume that
$\alpha\le\beta_F$. Then either $\frac{\pi}{3}<\beta_F$ or $\alpha\le\beta_F\le\frac{\pi}{3}$. First,
let us take a closer look of the case $\frac{\pi}{3}<\beta_F$. Reflect the point $\mathbf{o}_1$ about the plane of $F$ and label the point obtained by $\mathbf{o}_1'$. 

\begin{figure}[h!]
\begin{center}
\includegraphics[scale=0.15]{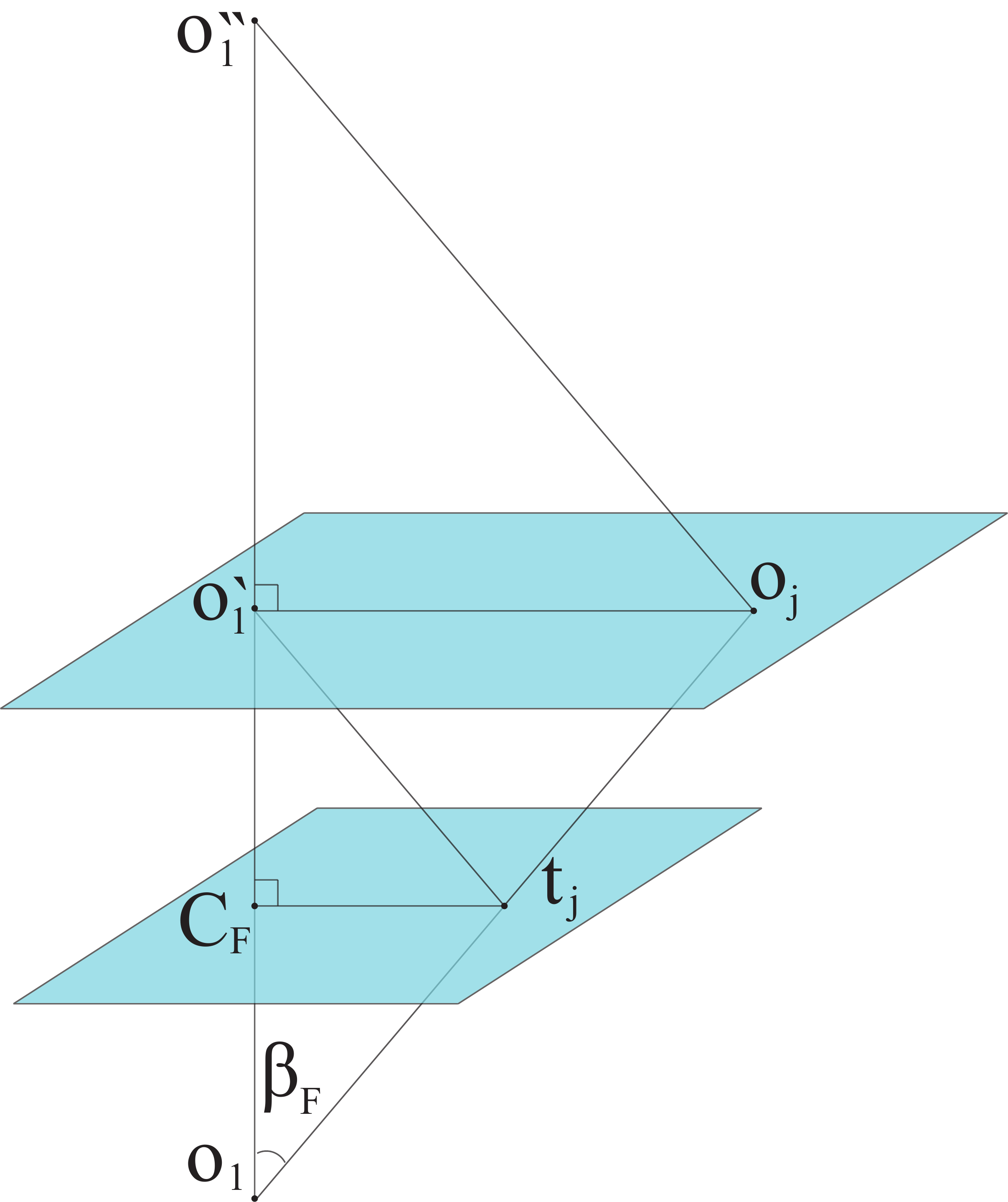}
\caption{The plane reflections to obtain $\mathbf{o}_{1}'$ and $\mathbf{o}_{1}''$.}
\label{fig:PlaneReflections}
\end{center}
\end{figure}

Clearly, the triangle $\triangle\mathbf{o}_1\mathbf{o}_1'\mathbf{o}_j$ is a right triangle with a right angle at $\mathbf{o}_1'$ and with an acute angle of measure $\beta_F$ at $\mathbf{o}_1$ for all $j\in I_F$. Then reflect the point $\mathbf{o}_1$ about $\mathbf{o}_1'$ and label the point obtained by $\mathbf{o}_1''$ furthermore, let $\mathbf{B}_1''$ denote the unit ball centered at $\mathbf{o}_1''$. As $\frac{\pi}{3}<\beta_F$ therefore ${\rm dist}(\mathbf{o}_1, \mathbf{o}_1'')<2$ and so, one can simply translate $\mathbf{B}_1''$ along the line $\mathbf{o}_1\mathbf{o}_1''$ away from $\mathbf{o}_1$ to a new position say, $\mathbf{B}_1'''$ such that it is tangent to $\mathbf{B}_1$. However, this would mean that $\mathbf{B}_1$ is tangent to $13$ non-overlapping unit balls namely, to $\mathbf{B}_1''', \mathbf{B}_2, \mathbf{B}_3, \dots , \mathbf{B}_{13}$, clearly contradicting to the well-known fact (\cite{SW}) that this number cannot be larger than $12$. Thus, we are left with the case when $\alpha\le\beta_F\le\frac{\pi}{3}$. By repeating the definitions of $\mathbf{o}_1'$, $\mathbf{o}_1''$, and $\mathbf{B}_1''$, the inequality $\beta_F\le\frac{\pi}{3}$ implies in a straightforward way that the $14$ unit balls $\mathbf{B}_1, \mathbf{B}_1'', \mathbf{B}_2, \mathbf{B}_3, \dots , \mathbf{B}_{13}$ form a packing in $\mathbb{E}^3$. Moreover, the inequality $\alpha\le\beta_F$ yields that
${\rm dist}(\mathbf{o}_1, \mathbf{o}_1'')\le4\cos\alpha=\frac{4}{\hat{r}}=2.51998... <2.52$. Finally, notice that the latter inequality contradicts to the following recent result of Hales \cite{H11}.

\begin{Theorem}\label{Hales}
Let $\mathbf{B}_1, \mathbf{B}_2, \dots , \mathbf{B}_{14}$ be $14$ different members of a packing of unit balls in $\mathbb{E}^3$. Assume that each ball of the family $\mathbf{B}_2, \mathbf{B}_3, \dots , \mathbf{B}_{13}$ touches $\mathbf{B}_1$. Then the distance between
the centers of $\mathbf{B}_1$ and $\mathbf{B}_{14}$ is at least $2.52$.
\end{Theorem}

This completes the proof of Lemma~\ref{B-B-S}.
\endproof  

Now, we are ready to prove Theorem~\ref{Bezdek-Hales-1}. First, we note that by projecting the faces $F$ of $\mathbf{T}$ from the center point $\mathbf{o}_1$ onto the sphere ${\rm bd}(\hat{\mathbf{B}}_1 )$ we get a tiling of ${\rm bd}(\hat{\mathbf{B}}_1 )$ into spherically convex polygons $\hat{F}$. Thus, it is sufficient to show that if $F$ is an arbitrary face of $\mathbf{T}$ with vertices $\mathbf{t}_j, j\in I_F\subset\{2, 3, \dots ,$ $ 13\}$, then its central projection $\hat{F}\subset{\rm bd}(\hat{\mathbf{B}}_1 )$ is covered by the closed balls $\hat{\mathbf{B}}_j, j\in I_F\subset\{2, 3, \dots ,$ $ 13\}$. Second, in order to achieve this it is sufficient to prove that the projection $\hat{\mathbf{c}}_F$ of the center $\mathbf{c}_F$ of the circumscribed circle of $F$ from the center point $\mathbf{o}_1$ onto the sphere ${\rm bd}(\hat{\mathbf{B}}_1 )$ is covered by each of the closed balls $\hat{\mathbf{B}}_j, j\in I_F\subset\{2, 3, \dots ,$ $ 13\}$. Indeed, if in the triangle $\triangle\mathbf{o}_1\mathbf{o}_j\hat{\mathbf{c}}_F$ the measure of the angle at $\mathbf{o}_1$ is denoted by $\beta_F$, then Lemma~\ref{B-B-S} implies in a straighforward way that $\beta_F<\alpha$. Hence, based on ${\rm dist}(\mathbf{o}_1,\mathbf{o}_j)=2$ and ${\rm dist}(\mathbf{o}_1,\hat{\mathbf{c}}_F)=\hat{r}$, a simple comparison of the triangle $\triangle\mathbf{o}_1\mathbf{o}_j\hat{\mathbf{c}}_F$ with the triangle $\triangle\mathbf{o}_1\mathbf{p}\mathbf{q}$ yields that ${\rm dist}(\mathbf{o}_j, \hat{\mathbf{c}}_F)<\hat{r}$ holds for all $j\in I_F\subset\{2, 3, \dots ,$ $ 13\}$, finishing the proof of Theorem~\ref{Bezdek-Hales-1}.  
\endproof

Next, let us take the union $\bigcup_{i=1}^n\left(\mathbf{c}_i+\hat{r}\mathbf{B}\right)$ of the closed balls $\mathbf{c}_1+\hat{r}\mathbf{B}, \mathbf{c}_2+\hat{r}\mathbf{B}, \dots , \mathbf{c}_n+\hat{r}\mathbf{B}$ of radii $\hat{r}$ centered at the points $\mathbf{c}_1, \mathbf{c}_2, \dots , \mathbf{c}_n$ in $\mathbb{E}^3$. 

\begin{Theorem}\label{Bezdek-Hales-2}
$$\frac{n{\rm vol}_3(\mathbf{B})}{{\rm vol}_3\left(\bigcup_{i=1}^n\left(\mathbf{c}_i+\hat{r}\mathbf{B}\right)\right)}<0.7547,$$
where ${\rm vol}_3(\cdot)$ refers to the $3$-dimensional volume of the corresponding set.
\end{Theorem}

\proof
First, partition $\bigcup_{i=1}^n\left(\mathbf{c}_i+\hat{r}\mathbf{B}\right)$ into truncated Voronoi cells as follows. Let $\mathbf{P}_i$ denote the Voronoi cell of the packing $\cal P$ assigned to $\mathbf{c}_i+\mathbf{B}$, $1\le i\le n$, that is, let $\mathbf{P}_i$ stand for the set of points of $\mathbb{E}^{3}$ that are not farther away from $\mathbf{c}_i$ than from any other $\mathbf{c}_j$ with $j\neq i, 1\le j\le n$. Then, recall the well-known fact (see for example, \cite{F64}) that the Voronoi cells $\mathbf{P}_i$, $1\le i\le n$ just introduced form a tiling of $\mathbb{E}^{3}$. Based on this it is easy to see that the truncated Voronoi cells $\mathbf{P}_i\cap (\mathbf{c}_i+\hat{r}\mathbf{B})$, $1\le i\le n$ generate a tiling of the non-convex container $\bigcup_{i=1}^n\left(\mathbf{c}_i+\hat{r}\mathbf{B}\right)$ for the packing $\cal P$. Second, as $\sqrt{2}<\hat{r}$ therefore the following very recent result of Hales \cite{H11} (see Lemma 9.13 on p. 228) applied to the truncated Voronoi cells $\mathbf{P}_i\cap (\mathbf{c}_i+\hat{r}\mathbf{B})$, $1\le i\le n$ implies the inequality of Theorem~\ref{Bezdek-Hales-2} in a straightforward way.

\begin{figure}[h!]
\begin{center}
\includegraphics[scale=0.28]{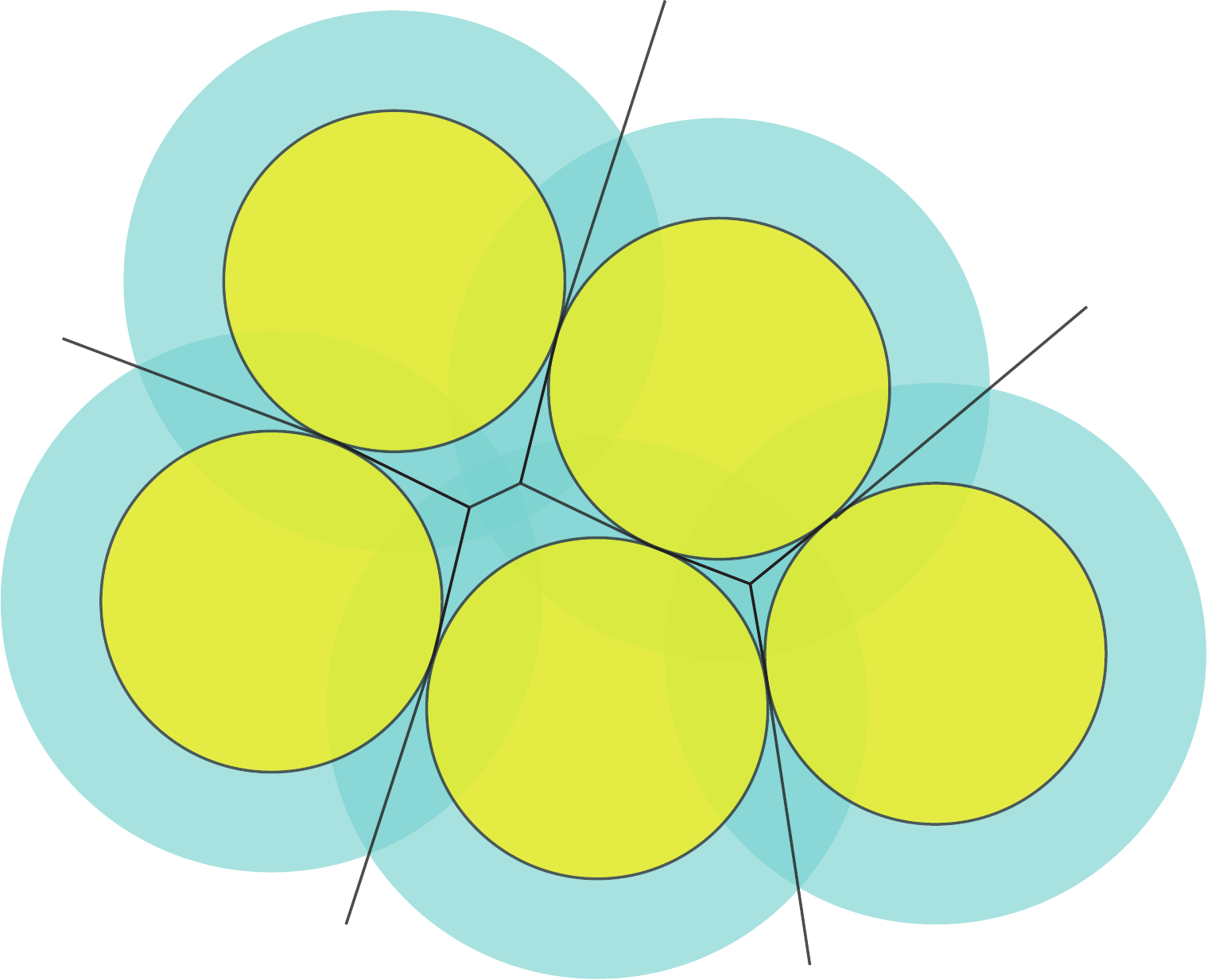}
\caption{Voronoi cells of a packing with yellow $\mathbf{c}_i+\mathbf{B}$'s and blue $\mathbf{c}_i+\hat{r}\mathbf{B}$'s.}
\label{fig:VoronoiCells}
\end{center}
\end{figure}

\begin{Theorem}\label{Hales-Voronoi}
Let $\cal F$ be an arbitrary (finite or infinite) family of non-overlap\-ping unit balls in $\mathbb{E}^3$ with the unit ball $\mathbf{B}$ centered at the origin $\mathbf{o}$ of $\mathbb{E}^3$ belonging to $\cal F$. Let $\mathbf{P}$ stand for the Voronoi cell of the packing $\cal F$ assigned to $\mathbf{B}$. Let $\mathbf{Q}$ denote a regular dodecahedron circumscribed $\mathbf{B}$ (having circumradius $\sqrt{3}\tan\frac{\pi}{5}=1.2584...$).  Finally, let $r:=\sqrt{2}=1.4142...$ and let $r\mathbf{B}$ denote the ball of radius $r$ centered at the origin $\mathbf{o}$ of $\mathbb{E}^3$. Then
$$\frac{{\rm vol}_3(\mathbf{B})}{{\rm vol}_3(\mathbf{P})}\le\frac{{\rm vol}_3(\mathbf{B})}{{\rm vol}_3(\mathbf{P}\cap r\mathbf{B})}\le\frac{{\rm vol}_3(\mathbf{B})}{{\rm vol}_3(\mathbf{Q})}<0.7547.$$
\end{Theorem}

This finishes the proof of Theorem~\ref{Bezdek-Hales-2}. 
\endproof

The well-known isoperimetric inequality \cite{Os78} applied to $\bigcup_{i=1}^n\left(\mathbf{c}_i+\hat{r}\mathbf{B}\right)$ yields 

\begin{Lemma}\label{isoperimetric-inequality}
$$36\pi{\rm vol}_3^2\left(\bigcup_{i=1}^n\left(\mathbf{c}_i+\hat{r}\mathbf{B}\right)\right)
\le{\rm svol}_2^3\left({\rm bd}\left(\bigcup_{i=1}^n\left(\mathbf{c}_i+\hat{r}\mathbf{B}\right)\right)\right),$$
where ${\rm svol}_2(\cdot)$ refers to the $2$-dimensional surface volume of the corresponding set.
\end{Lemma}

Thus, Theorem~\ref{Bezdek-Hales-2} and Lemma~\ref{isoperimetric-inequality} generate the following inequality.

\begin{Corollary}\label{lower-bound-for-surface-area-in-3D}
$$15.159805n^{\frac{2}{3}}<15.15980554...n^{\frac{2}{3}}=\frac{4\pi}{(0.7547)^{\frac{2}{3}}}n^{\frac{2}{3}}< {\rm svol}_2\left({\rm bd}\left(\bigcup_{i=1}^n\left(\mathbf{c}_i+\hat{r}\mathbf{B}\right)\right)\right).$$
\end{Corollary}

Now, assume that $\mathbf{c}_i+\mathbf{B}\in {\cal P}$ is tangent to $\mathbf{c}_j+\mathbf{B}\in {\cal P}$ for all $j\in T_i$, where $T_i\subset\{1, 2, \dots , n\}$ stands for the family of indices $1\le j\le n$ for which ${\rm dist}(\mathbf{c}_i, \mathbf{c}_j)=2$. Then let
$\hat{S}_i:={\rm bd}(\mathbf{c}_i+\hat{r}\mathbf{B})$ and let $\hat{\mathbf{c}}_{ij}$ be the intersection of the line segment $\mathbf{c}_i\mathbf{c}_j$ with $\hat{S}_i$ for all $j\in T_i$. Moreover, let $C_{\hat{S}_i}(\hat{\mathbf{c}}_{ij}, \frac{\pi}{6})$ (resp., $C_{\hat{S}_i}(\hat{\mathbf{c}}_{ij}, \alpha)$) denote the open spherical cap of $\hat{S}_i$ centered at $\hat{\mathbf{c}}_{ij}\in \hat{S}_i$ having angular radius $\frac{\pi}{6}$ (resp., $\alpha$ with $0<\alpha<\frac{\pi}{2}$ and $\cos\alpha=\frac{1}{\hat{r}}$). Clearly, the family $\{C_{\hat{S}_i}(\hat{\mathbf{c}}_{ij}, \frac{\pi}{6}), j\in T_i\}$ consists of pairwise disjoint open spherical caps of $\hat{S}_i$; moreover,

\begin{equation}\label{Bezdek-estimate-I}
\frac{\sum_{j\in T_i}{\rm svol}_2\left(C_{\hat{S}_i}(\hat{\mathbf{c}}_{ij}, \frac{\pi}{6})\right)}{{\rm svol}_2\left(\cup_{j\in T_i}C_{\hat{S}_i}(\hat{\mathbf{c}}_{ij}, \alpha)\right)}=
\frac{\sum_{j\in T_i}{\rm Sarea}\left(C(\mathbf{u}_{ij}, \frac{\pi}{6})\right)}{{\rm Sarea}\left(\cup_{j\in T_i}C(\mathbf{u}_{ij}, \alpha)\right)},
\end{equation}

\noindent where $\mathbf{u}_{ij}:=\frac{1}{2}(\mathbf{c}_j-\mathbf{c}_i)\in \mathbb{S}^2:={\rm bd}(\mathbf{B})$ and $C(\mathbf{u}_{ij}, \frac{\pi}{6})\subset \mathbb{S}^2$ (resp., $C(\mathbf{u}_{ij}, \alpha)\subset \mathbb{S}^2$) denotes the open spherical cap of $\mathbb{S}^2$ centered at $\mathbf{u}_{ij}$ having angular radius $\frac{\pi}{6}$ (resp., $\alpha$)
and where ${\rm Sarea}(\cdot)$ refers to the spherical area measure on $\mathbb{S}^2$. Now, Moln\'ar's density bound (Satz I in \cite{Mo65}) implies that
\begin{equation}\label{Bezdek-estimate-II}
\frac{\sum_{j\in T_i}{\rm Sarea}\left(C(\mathbf{u}_{ij}, \frac{\pi}{6})\right)}{{\rm Sarea}\left(\cup_{j\in T_i}C(\mathbf{u}_{ij}, \alpha)\right)}<0.89332\ .
\end{equation}

In order to estimate $${\rm svol}_2\left({\rm bd}\left(\bigcup_{i=1}^n\left(\mathbf{c}_i+\hat{r}\mathbf{B}\right)\right)\right)$$ from above let us assume that $m$ members of ${\cal P}$ have $12$ touching neighbours in ${\cal P}$ and $k$ members of ${\cal P}$ have at most $9$ touching neighbours in ${\cal P}$. Thus, $n-m-k$ members of ${\cal P}$ have either $10$ or $11$ touching neighbours in ${\cal P}$. (Here we have used the well-known fact that $\tau_3=12$, that is, no member of ${\cal P}$ can have more than $12$ touching neighbours.) Without loss of generality we may assume that $4\le k\le n-m$. 

First, we note that ${\rm Sarea}\left(C(\mathbf{u}_{ij}, \frac{\pi}{6})\right)=2\pi(1-\cos\frac{\pi}{6})=2\pi(1-\frac{\sqrt{3}}{2})$ and ${\rm svol}_2\left(C_{\hat{S}_i}(\hat{\mathbf{c}}_{ij}, \frac{\pi}{6})\right)=2\pi(1-\frac{\sqrt{3}}{2})\hat{r}^2$. Second, recall Theorem~\ref{Bezdek-Hales-1} according to which if a member of ${\cal P}$ say, $\mathbf{c}_i+\mathbf{B}$ has exactly $12$ touching neighbours in ${\cal P}$, then $\hat{S}_i\subset \bigcup_{j\in T_i}(\mathbf{c}_j+\hat{r}\mathbf{B})$. These facts together with (\ref{Bezdek-estimate-I}) and (\ref{Bezdek-estimate-II}) imply the following estimate.

\begin{Corollary}\label{upper-bound-for-surface-area-in-3D}
${\rm svol}_2\left({\rm bd}\left(\bigcup_{i=1}^n\left(\mathbf{c}_i+\hat{r}\mathbf{B}\right)\right)\right)< 
\frac{24.53902}{3} (n-m-k)+24.53902k \ .
$
\end{Corollary}

\proof
$${\rm svol}_2\left({\rm bd}\left(\bigcup_{i=1}^n\left(\mathbf{c}_i+\hat{r}\mathbf{B}\right)\right)\right)$$
$$<\left(4\pi\hat{r}^2-\frac{10\cdot2\pi(1-\frac{\sqrt{3}}{2})\hat{r}^2}{0.89332}\right)(n-m-k)
+\left(4\pi\hat{r}^2-\frac{3\cdot2\pi(1-\frac{\sqrt{3}}{2})\hat{r}^2}{0.89332}\right)k$$
$$<7.91956(n-m-k)+24.53902k
<\frac{24.53902}{3}(n-m-k)+24.53902k\ .$$ 
\endproof

Hence, Corollary~\ref{lower-bound-for-surface-area-in-3D} and Corollary~\ref{upper-bound-for-surface-area-in-3D} yield in a straightforward way that

\begin{equation}\label{Bezdek-estimate-III}
1.85335n^{\frac{2}{3}}-3k<n-m-k \ .
\end{equation}

Finally, as the number $C(n)$ of touching pairs in ${\cal P}$ is obviously at most $$
\frac{1}{2}\left( 12n-(n-m-k)-3k\right)\ ,
$$
therefore (\ref{Bezdek-estimate-III}) implies that
$$
C(n)\le \frac{1}{2}\left( 12n-(n-m-k)-3k\right)
< 6n-0.926675n^{\frac{2}{3}}
<6n-0.926n^{\frac{2}{3}},
$$
finishing the proof of $(i)$ in Theorem~\ref{improved-estimate}.

\section{Upper bounds for Lattice Packings}
\subsection{Proof of (ii) in Theorem~\ref{improved-estimate}}

Let us imagine that we generate packings of $n$ unit balls in $\mathbb{E}^{3}$ in such a special way that each and every center of the $n$ unit balls chosen, is a lattice point of some fixed lattice $\Lambda$ (resp., of the face-centered cubic lattice $\Lambda_{fcc}$) with shortest non-zero lattice vector of length $2$. (Here, a lattice means a (discrete) set of points having position vectors that are integer linear combinations of three fixed linearly independent vectors of $\mathbb{E}^{3}$.) Then let $C_{\Lambda}(n)$ (resp., $C_{fcc}(n)$) denote the largest possible number of touching pairs for all packings of $n$ unit balls obtained in this way. In order to prove $(ii)$ in Theorem~\ref{improved-estimate} it is sufficient to show that $C_{\Lambda}(n)\le C_{fcc}(n)$ and recall from \cite{B12} that $C_{fcc}(n)< 6n-\frac{3\sqrt[3]{18\pi}}{\pi}n^{\frac{2}{3}}=6n-3.665\dots n^{\frac{2}{3}}$. So, we are left to show that $C_{\Lambda}(n)\le C_{fcc}(n)$. The details are as follows.

Recall Voronoi's theorem (see \cite{CS92}) according to which every $3$-dimensional lattice is of the {\it first kind} i.e., it has an {\it obtuse superbase}. Thus, for the lattice $\Lambda$ (resp., $\Lambda_{fcc}$) we have a set of vectors $\mathbf{v}_0, \mathbf{v}_1, \mathbf{v}_2, \mathbf{v}_3$ (resp., $\mathbf{w}_0, \mathbf{w}_1, \mathbf{w}_2, \mathbf{w}_3$ ) such that $\mathbf{v}_1, \mathbf{v}_2, \mathbf{v}_3$ (resp., $\mathbf{w}_1, \mathbf{w}_2, \mathbf{w}_3$) is an integral basis for $\Lambda$ (resp., $\Lambda_{fcc}$) and $\mathbf{v}_0+\mathbf{v}_1+\mathbf{v}_2+\mathbf{v}_3=\mathbf{o}$ (resp., $\mathbf{w}_0+\mathbf{w}_1+\mathbf{w}_2+\mathbf{w}_3=\mathbf{o}$), and in addition $\mathbf{v}_i\cdot\mathbf{v}_j\le 0$ (resp., $\mathbf{w}_i\cdot\mathbf{w}_j\le 0$) for all $i,j=0,1,2,3$, $i\neq j$. Here ${}\cdot{}$ refers to the standard inner product of $\mathbb{E}^3$. Let $\mathbf{P}$ (resp., $\mathbf{Q}$) denote the Voronoi cell for the origin $\mathbf{o}\in \Lambda$ (resp., $\mathbf{o}\in \Lambda_{fcc}$) consisting of points of $\mathbb{E}^3$ that are at least as close to $\mathbf{o}$ as to any other lattice point of $\Lambda$ (resp., $\Lambda_{fcc}$). A vector $\mathbf{v}\in \Lambda$ (resp., $\mathbf{w}\in \Lambda_{fcc}$) is called a {\it strict Voronoi vector} of $\Lambda$ (resp., $\Lambda_{fcc}$) if the plane $\{\mathbf{x}\in \mathbb{E}^3\ | \ \mathbf{x}\cdot\mathbf{v}=\frac{1}{2}\mathbf{v}\cdot\mathbf{v}\} $ (resp., $\{\mathbf{x}\in \mathbb{E}^3\ | \ \mathbf{x}\cdot\mathbf{w}=\frac{1}{2}\mathbf{w}\cdot\mathbf{w}\} $ ) intersects $\mathbf{P}$ (resp., $\mathbf{Q}$) in a face. We need the following claim proved in \cite{CS92}. The list of $14$ lattice vectors of $\Lambda$ (resp., $\Lambda_{fcc}$) consisting of $$\pm\mathbf{v}_1,\pm(\mathbf{v}_0+\mathbf{v}_1), \pm(\mathbf{v}_1+\mathbf{v}_2), \pm(\mathbf{v}_1+\mathbf{v}_3),$$
$$\pm(\mathbf{v}_0+\mathbf{v}_1+\mathbf{v}_2), \pm(\mathbf{v}_0+\mathbf{v}_1+\mathbf{v}_3), \pm(\mathbf{v}_1+\mathbf{v}_2+\mathbf{v}_3)  $$
$${\rm (resp.,}\pm\mathbf{w}_1,\pm(\mathbf{w}_0+\mathbf{w}_1), \pm(\mathbf{w}_1+\mathbf{w}_2), \pm(\mathbf{w}_1+\mathbf{w}_3),$$
$$\pm(\mathbf{w}_0+\mathbf{w}_1+\mathbf{w}_2), \pm(\mathbf{w}_0+\mathbf{w}_1+\mathbf{w}_3), \pm(\mathbf{w}_1+\mathbf{w}_2+\mathbf{w}_3)  {\rm )}$$  
includes all the strict Voronoi vectors of $\Lambda$ (resp., $\Lambda_{fcc}$). As is well known (and in fact, it is easy check) at most 12 (resp., exactly 12) of the above 14 vectors has length $2$ and the others are of length strictly greater than $2$. Thus, it follows that without loss of generality we may assume that whenever $\mathbf{v}_i\cdot\mathbf{v}_i=4$ holds we have $\mathbf{w}_i\cdot\mathbf{w}_i=4$ as well. This implies the exisctence of a map $f : \Lambda\rightarrow\Lambda_{fcc}$ with the property that if ${\rm dist}(\mathbf{x},\mathbf{y})=2$ with $\mathbf{x}, \mathbf{y}\in \Lambda$, then also ${\rm dist}(f(\mathbf{x}),f(\mathbf{y}))=2$ holds. Indeed, $f$ can be defined via $f(\alpha\mathbf{v}_1+\beta\mathbf{v}_2+\gamma\mathbf{v}_3)=
\alpha\mathbf{w}_1+\beta\mathbf{w}_2+\gamma\mathbf{w}_3$ with $\alpha, \beta, \gamma $ being arbitrary integers. As a result we get the following: if ${\cal P}$ is a packing of $n$ unit balls with centers $\mathbf{c}_1, \mathbf{c}_2, \dots , \mathbf{c}_n\in \Lambda$, then the packing ${\cal P}_f$ of $n$ unit balls centered at the points $f(\mathbf{c}_1), f(\mathbf{c}_2), \dots , f(\mathbf{c}_n)\in \Lambda_{fcc}$ possesses the property that $C({\cal P})\le C({\cal P}_f)$, where $C({\cal P})$ (resp., $C({\cal P}_f)$) stands for the number of touching pairs in $\cal P$ (resp., ${\cal P}_f$). Thus, indeed, $C_{\Lambda}(n)\le C_{fcc}(n)$ finishing the proof of $(ii)$ in Theorem~\ref{improved-estimate}.

\subsection{Proof of $(ii)$ in Theorem~\ref{main-theorem}}

Based on the previous subsection, it is sufficient to prove the estimate in question on the touching triplets (resp., quadruples) when the packing $\cal P$ of $n$ unit balls in $\mathbb{E}^{3}$ is given in such a special way that each center is a lattice point of the face-centered cubic lattice $\Lambda_{fcc}$ with shortest non-zero lattice vector of length $2$. Then we take the \textit{contact graph} $G(\mathcal{P})$ of $\mathcal{P}$ with vertices identical to the center points of the unit balls in $\cal P$ and with edges between two vertices if the corresponding two unit balls of $\cal P$ touch each other. Clearly, a touching triplet (resp., quadruple) in $\cal P$ corresponds to a regular triangle (resp., regular tetrahedron) of edge length $2$ in $G(\mathcal{P})$. Using the symmetries of $\Lambda_{fcc}$, it is easy to check that at most $24$ (resp., $8$) regular triangles (resp., tetrahedra) of edge length $2$ can have a vertex in common in  $G(\mathcal{P})$.  Thus, a straightforward counting argument shows that the number of touching triplets (resp., quadruples) in  $\mathcal{P}$ is at most $\frac{24n}{3}=8n$ (resp., $\frac{8n}{4}=2n$), finishing the proof of $(ii)$ in Theorem~\ref{main-theorem}.

\section{Proof of (i) in Theorem~\ref{main-theorem} using Theorem~\ref{spherical-main-theorem}}

Let $\mathcal{P}$ be an arbitrary packing of $n$ unit balls in $\mathbb{E}^3$. Let $\textbf{B}$ stand for the unit ball centered at the origin of $\mathbb{E}^3$ and let $\mathcal{P}=\{\mathbf{c}_{1}+\textbf{B},..., \mathbf{c}_{n}+\textbf{B}\}$. Then, we take the \textit{contact graph} $G(\mathcal{P})$ of $\mathcal{P}$ whose vertices are $\mathbf{c}_{1},...,\mathbf{c}_{n}$ with an edge connected between two vertices if $\mathbf{c}_{i} + \textbf{B}$ and $\mathbf{c}_{j} + \textbf{B}$ touch each other, i.e., ${\rm dist}(\mathbf{c}_{i}, \mathbf{c}_{j})=2$. Every touching triplet $\mathbf{c}_{i} + \textbf{B}, \mathbf{c}_{j} + \textbf{B}$, and $\mathbf{c}_{k} + \textbf{B}$ (resp. touching quadruple $\mathbf{c}_{i}+\textbf{B}, \mathbf{c}_{j}+\textbf{B}, \mathbf{c}_{k}+\textbf{B}$, and $\mathbf{c}_{l}+\textbf{B}$) of $\mathcal{P}$ corresponds to a regular triangle spanned by $\mathbf{c}_{i}, \mathbf{c}_{j}$, and $\mathbf{c}_{k}$ (resp. regular tetrahedron spanned by $\mathbf{c}_{i}, \mathbf{c}_{j}, \mathbf{c}_{k}$, and $\mathbf{c}_{l}$) of edge length 2 in $G(\mathcal{P})$.

\begin{figure}[h!]
\begin{center}
\includegraphics[scale=0.28]{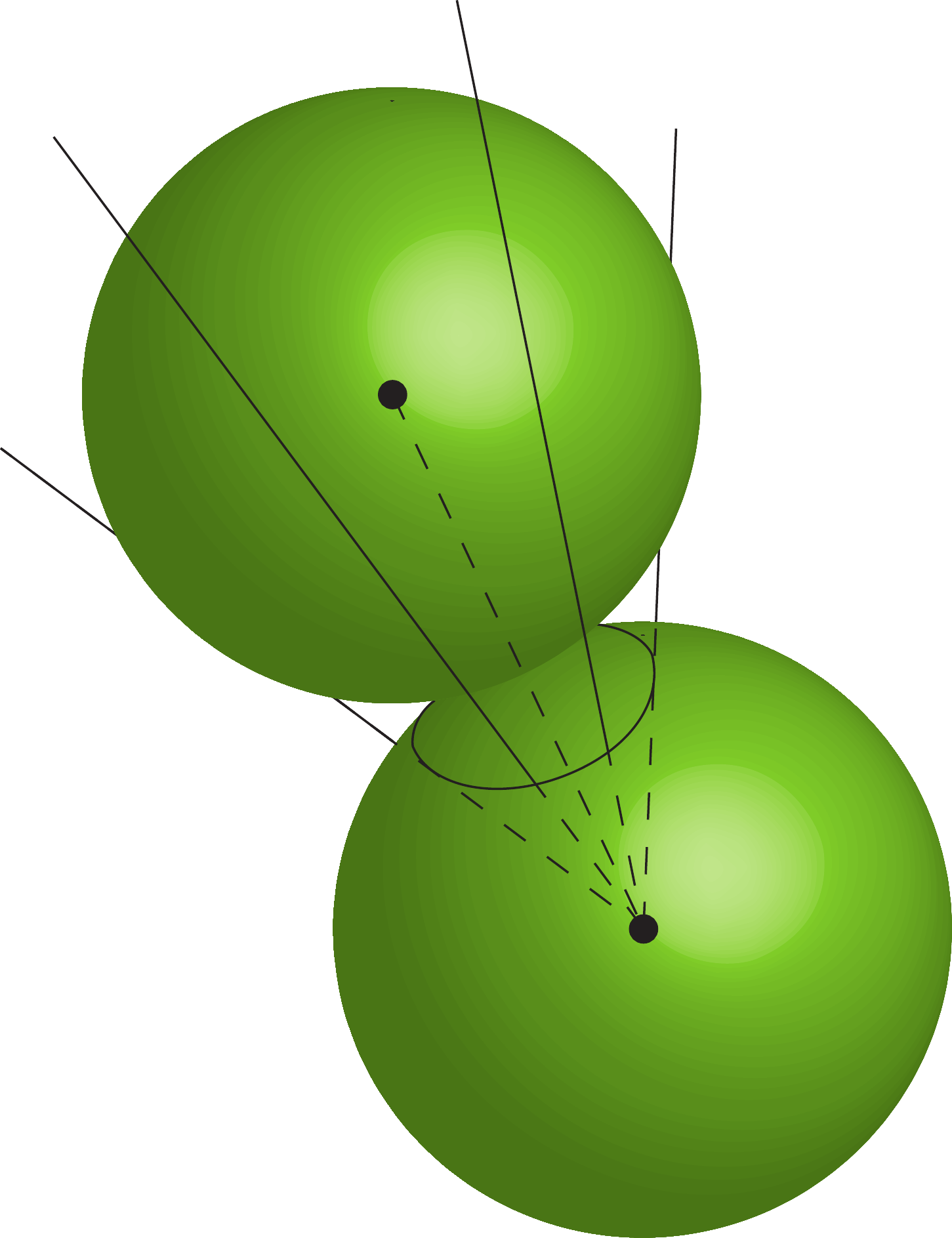}
\caption{Projecting a spherical cap between two unit balls.}
\label{fig:SphericalCap}
\end{center}
\end{figure}

\begin{Lemma}\label{core-lemma}
The number of regular triangles (resp., regular tetrahedra) of edge length 2 sharing a vertex in $G(\mathcal{P})$ is at most 25 (resp., 11).
\end{Lemma}
\begin{proof}
Let the unit balls $\mathbf{c}_{i} + \textbf{B}$ and $\mathbf{c}_{j} + \textbf{B}$ of $\mathcal{P}$ be touching. Then the central projection of $\mathbf{c}_{j} + \textbf{B}$ from $\mathbf{c}_{i}$ onto the boundary of $\mathbf{c}_{i} + \textbf{B}$ is a spherical cap of angular radius $\pi/6$, as seen in Figure \ref{fig:SphericalCap}. By projecting each unit ball $\mathbf{c}_{j} + \textbf{B}$ of $\mathcal{P}$ that touches $\mathbf{c}_{i} + \textbf{B}$ onto the boundary of $\mathbf{c}_{i} + \textbf{B}$, we get a packing of spherical caps of angular radius of $\pi/6$ on the boundary of $\mathbf{c}_{i} + \textbf{B}$. Therefore, Theorem~\ref{spherical-main-theorem} finishes the proof of Lemma~\ref{core-lemma}.
\end{proof}

We now prove the desired upper bound on the number of the regular triangles of edge length 2 in $G(\mathcal{P})$. By Lemma~\ref{core-lemma}, we have that there are at most 25 regular triangles in the contact graph $G(\mathcal{P})$ sharing a vertex, so we count $25n$ touching triplets in $\mathcal{P}$. Yet, since each triplet is counted at each of the three vertices of the regular triangles, we divide by three to avoid over counting, thus leading to the bound of at most $\frac{25n}{3}$ touching triplets in $\mathcal{P}$. Finally, using Lemma~\ref{core-lemma} again, a similar counting argument yields that the number of regular tetrahedra of edge length 2 in $G(\mathcal{P})$ is at most $\frac{11n}{4}$.

\section{The Polygon Lemmas on $\mathbb{S}^2$ for Theorem~\ref{spherical-main-theorem}}

On $\mathbb{S}^2$, we take a point set $X = \{\mathbf{x}_{1},...,\mathbf{x}_{N}\}$ with minimum spherical distance $\pi/3$ between any two points in $X$. The solution to the Newton-Gregory problem (of determining the maximum number of non-overlapping unit balls which can touch a fixed unit ball) by Sch\"{u}tte and van der Waerden \cite{SW} implies that $N \leq 12$. Taking the Delaunay triangulation $\mathcal{D}_{X}$ of $X$ on $\mathbb{S}^2$, we notice that we can classify the triangles based on how many times a side length of greater than $\pi/3$ occurs. (For many of the basic properties of Delaunay triangulations we refer the interested reader to \cite{F64} as well as \cite{Ma02}.) In fact, we say that an irregular triangle in $\mathcal{D}_{X}$ is of type $R$ where $R$ is the number of side lengths of the triangle greater than $\pi/3$. We will reserve the term \textit{regular triangle} for the type 0 triangles which have side lengths all equal to $\pi/3$. For the sake of completeness we note that if two points of $X$ lie at (spherical) distance $\frac{\pi}{3}$ from each other, then the geodesic line segment (i.e., great circular arc of length $\frac{\pi}{3}$) connecting them is an edge of $\mathcal{D}_{X}$ on $\mathbb{S}^2$. Based on this we note also that the method described in this section as well as in the following one is quite general and applies to any triangulation of $X$ on $\mathbb{S}^2$ that possesses the above mentioned edge property.

\begin{figure}[h!]
\begin{center}
\includegraphics[scale=0.15]{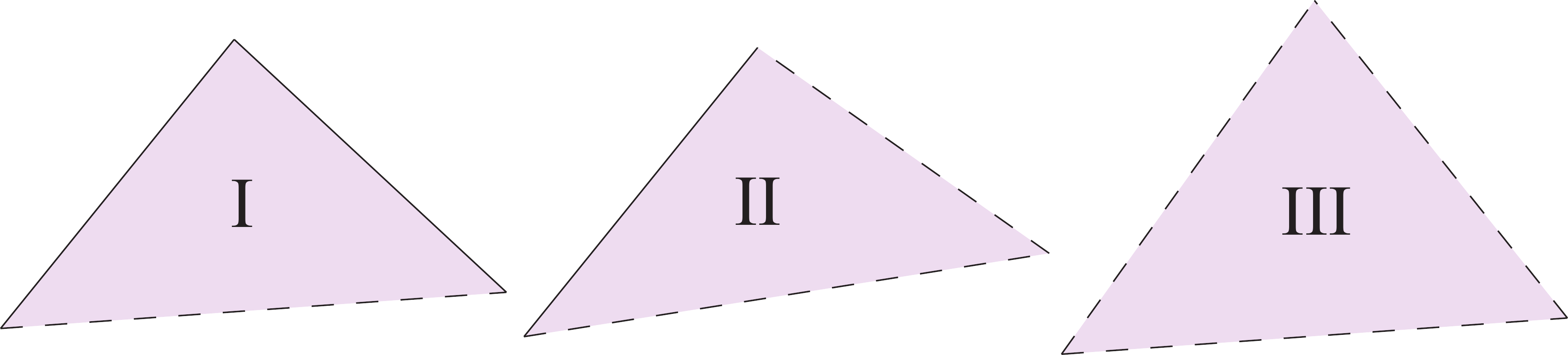}
\caption{Three types of irregular triangles, types I, II, and III from left to right. Dashed sides represent side lengths of greater than $\pi/3$ and non-dashed sides represent side lengths of exactly $\pi/3$.}
\label{fig:IrregularTriangles}
\end{center}
\end{figure}

We now prove three lemmas, the Quadrilateral Lemma, the Pentagon Lemma, and the Hexagon Lemma, which are needed for our proof of Theorem~\ref{spherical-main-theorem}. We note that, in what follows, the angles of a regular triangle (of side length $\frac{\pi}{3}$) have radian measure $\arccos(1/3)$ ($=70.528\dots ^{\circ}$).

Let $C_4$ denote a spherical quadrilateral of side lengths $\pi/3$ which triangulates into two irregular triangles of type $I$. The Quadrilateral Lemma ensures that $C_4$ cannot exist in $\mathcal{D}_{X}$ when there are two adjacent vertices of $C_4$ say, $\mathbf{v}$ and $\mathbf{w}$ such that all of the Delaunay triangles of $\mathcal{D}_{X}$, not in $C_4$, having $\textbf{v}$ or $\mathbf{w}$ as a vertex are regular.


\begin{figure}[H]
\begin{center}
\includegraphics[scale=0.3]{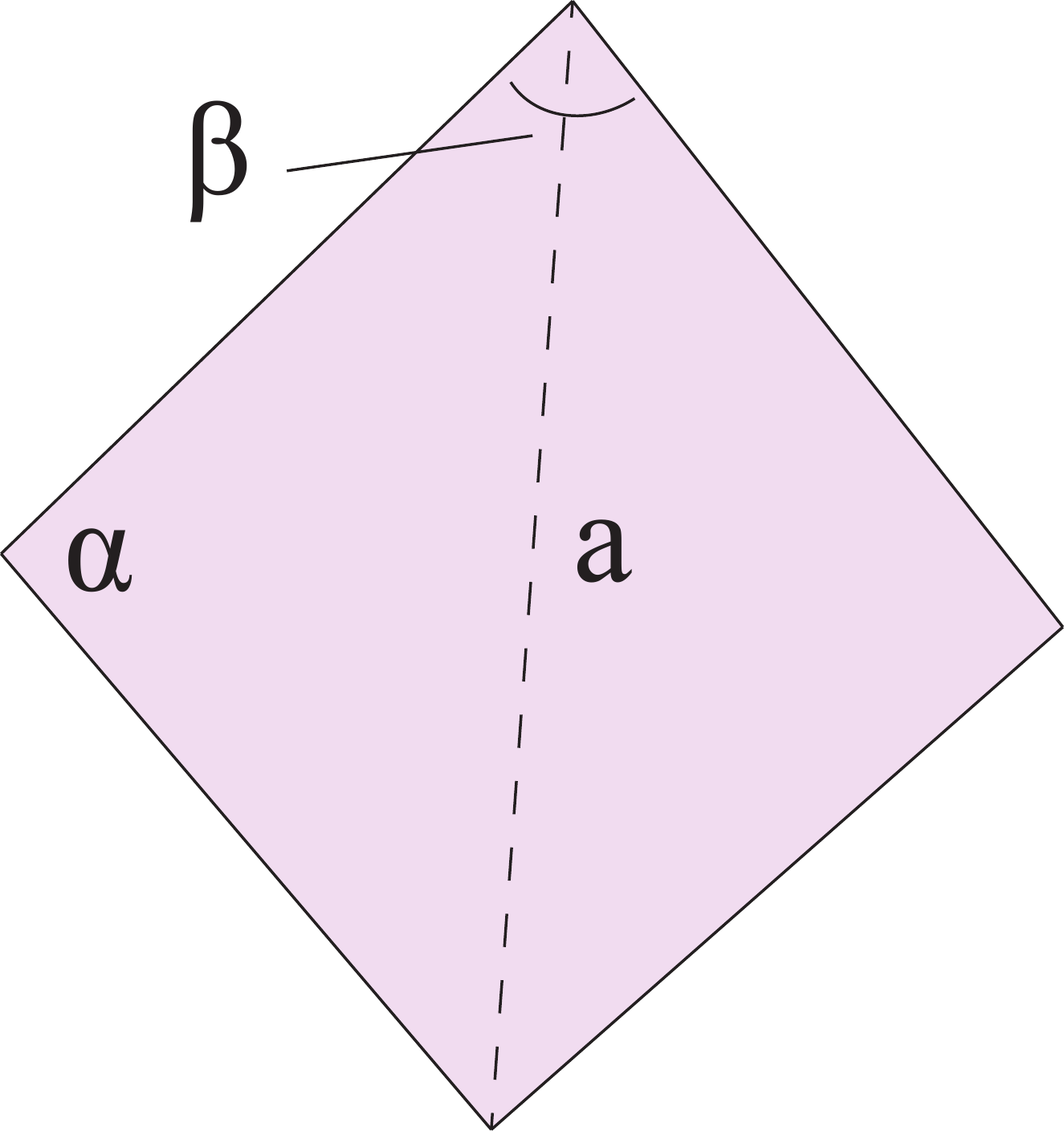}
\caption{A triangulated $C_4$ quadrilateral with side lengths of $\pi/3$.}
\label{fig:QuadrilateralLemma}
\end{center}
\end{figure}

\begin{Lemma}(Quadrilateral Lemma) \\
Let $\alpha$ and $\beta$ denote the internal angles of $C_4$ subtended at adjacent vertices of $C_4$ as shown in Figure \ref{fig:QuadrilateralLemma}. If $\alpha \in \{2\pi - k\arccos(1/3) \; | \; k=1,2,3\}$, then $C_4$ cannot exist in $\mathcal{D}_{X}$ and if $\alpha = 2\pi - 4\arccos(1/3)$, then $\beta \notin \{2\pi - k\arccos(1/3) \; | \; k=1,2,3,4\}$.
\end{Lemma}
\begin{proof}
If $\alpha=2\pi-\arccos(1/3)$ or $\alpha=2\pi-2\arccos(1/3)$, then $C_4$ is non-convex
with $\alpha>\pi$, a contradiction. So, either $\alpha = 2\pi - 3\arccos(1/3)$ or $\alpha = 2\pi - 4\arccos(1/3)$. By the (first) law of cosines (see for example \cite{Ra06}),
$$a = \arccos\left(\frac{1+3\cos\alpha}{4}\right)$$
From the symmetry of $C_4$ about the diagonal and the first law of cosines,
$$\beta/2 = \arccos\left(\frac{1-\cos a}{\sqrt{3}\sin a}\right)$$
We now consider our possible cases for varying $\alpha$,
\begin{center}
Table 1. Cases for the Quadrilateral Lemma \\
\begin{tabular}{| c | c | c | c |} 
\hline
Cases & $\alpha$ & $a$ & $\beta$ \\ \hline
(1) & 1.359 & 1.151 & 2.373 \\ \hline
(2) & 2.590 & 1.970 & 1.029 \\ \hline
\end{tabular}
\end{center}
In Case (2) we have that $\beta = 1.029 < \pi/3$ and so, the corresponding two vertices of $C_4$ lie closer to each other than $\frac{\pi}{3}$, a contradiction. In Case (1), which is realizable, we have
$\beta \notin \{2\pi-k\arccos(1/3) \; | \; k=1,2,3,4\}$.
\end{proof}

Let $C_5$ denote a spherical pentagon of side lengths $\pi/3$ which triangulates into two irregular triangles of type I and one irregular triangle of type II, as shown in Figure \ref{fig:PentagonLemma}. The Pentagon Lemma ensures that $C_5$ cannot exist in $\mathcal{D}_{X}$ when all of the Delaunay triangles of $\mathcal{D}_{X}$, not in $C_5$, sharing a vertex in common with $C_5$ are regular.

\begin{figure}[H]
\begin{center}
\includegraphics[scale=0.25]{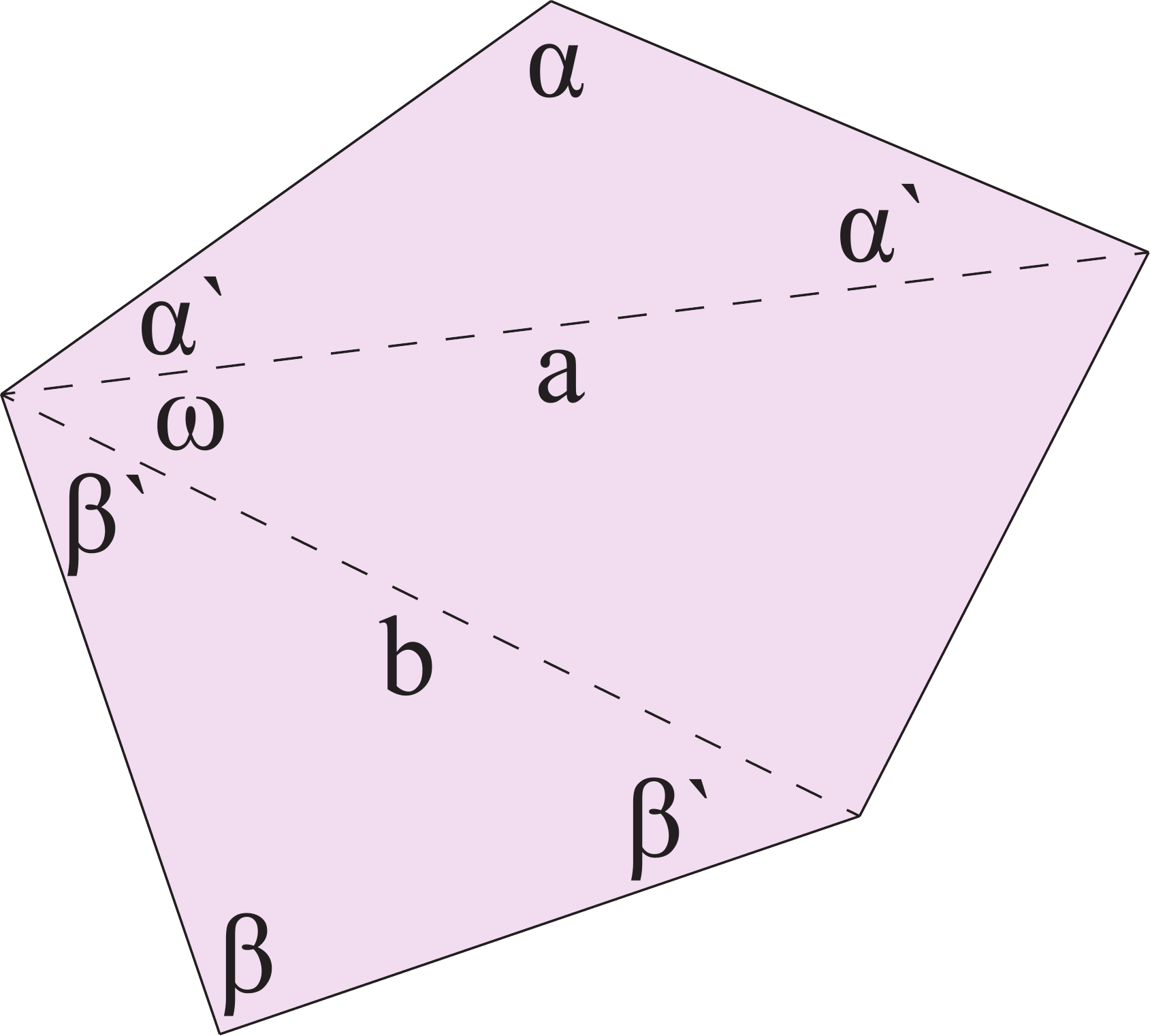}
\caption{A triangulated $C_5$ pentagon with side lengths of $\pi/3$.}
\label{fig:PentagonLemma}
\end{center}
\end{figure}

\begin{Lemma}(Pentagon Lemma) \\
Let $\alpha$ and $\beta$ denote the non-adjacent internal angles of $C_5$ subtended at the vertices of the type I triangles of $C_5$, and let $\omega$ denote the interior angle of the type II triangle of $C_5$ opposite to the only side of the type II triangle with length $\pi/3$, as shown in Figure \ref{fig:PentagonLemma}. If $\alpha, \beta \in \{2\pi-k\arccos(1/3) \; | \; k=3,4\}$, then $\alpha' + \beta' + \omega \notin \{2\pi - k\arccos(1/3) \; | \; k=1,2,3,4\}$, where $\alpha'$ and $\beta'$ are the internal angles of the two type I irregular triangles not equal to $\alpha$ and $\beta$, respectively.
\end{Lemma}
\begin{proof}
Letting $\alpha, \beta \in \{2\pi-k\arccos(1/3) \; | \; k=3,4\}$, we can consider three cases: (1) $\alpha=\beta=2\pi - 4\arccos(1/3)$, (2) $\alpha=\beta=2\pi - 3\arccos(1/3)$, and (3) $\alpha \neq \beta$. Letting $a$ and $b$ denote the side lengths opposite to $\alpha$ and $\beta$ in the corresponding type I triangles, we can compute the side lengths $a$ and $b$ and the internal angles $\alpha',\beta',\omega$ by the first law of cosines,
\begin{equation*}
a = \arccos\left(\frac{1+3\cos\alpha}{4}\right) \;\;\;\;\;\; b = \arccos\left(\frac{1+3\cos\beta}{4}\right)
\end{equation*}

\begin{equation*}
\alpha'=\arccos\left(\frac{1-\cos a}{\sqrt{3}\sin a}\right) \;\;\;\;\;\;
\beta'=\arccos\left(\frac{1-\cos b}{\sqrt{3}\sin b}\right)
\end{equation*}

\begin{equation*}
\omega = \arccos\left(\frac{\cos(\pi/3)-\cos a \cos b}{\sin a \sin b}\right)
\end{equation*}

We now consider our possible cases for varying $\alpha$ and $\beta$,
\begin{center}
Table 2. Cases for the Pentagon Lemma
\begin{tabular}{| c | c | c | c | c | c | c | c | c |} 
\hline
Cases & $\alpha$ & $\beta$ & $a$ & $b$ & $\alpha'$ & $\beta'$ & $\omega$ & $\alpha' + \beta' + \omega$ \\ \hline
(1) & 1.359 & 1.359 & 1.151 & 1.151 & 1.1867 & 1.1867 & 1.158 & 3.532 \\ \hline
(2) & 2.590 & 2.590 & 1.970 & 1.970 & 0.5148 & 0.5148 & 1.147 & 2.176 \\ \hline
(3) & 1.359 & 2.590 & 1.151 & 1.970 & 1.1867 & 0.5148 & 0.671 & 2.373 \\ \hline
\end{tabular}
\end{center}
Therefore, we have that in each case,
$\alpha' + \beta' + \omega \notin \{2\pi - k\arccos(1/3) \; | \; k=1,2,3,4\}$.
\end{proof}

Let $C_6$ denote a spherical hexagon of side lengths $\pi/3$ which triangulates into three irregular triangles of type I and one irregular triangle of type III, as shown in Figure \ref{fig:HexagonLemma1}. Let $C_6'$ denote a spherical hexagon of side lengths $\pi/3$ which triangulates into two irregular triangles of type I and two irregular triangles of type II, as shown in Figure \ref{fig:HexagonLemma2}. Let $C_6''$ denote a spherical hexagon of side lengths $\pi/3$ which triangulates into two irregular triangles of type I and two irregular triangles of type II, as shown in Figure \ref{fig:HexagonLemma3}. The Hexagon Lemma ensures that neither $C_6$, $C_6'$ nor $C_6''$ can exist in $\mathcal{D}_{X}$ when all of the Delaunay triangles of $\mathcal{D}_{X}$, not in $C_6$, $C_6'$ or $C_6''$, sharing a vertex in common with either $C_6$, $C_6'$, or $C_6''$ are regular. The cases of $C_6$, $C_6'$, and $C_6''$ hexagons are the only possible Delaunay triangulations of a hexagon (into irregular ones), so they are the only cases we need to consider for a spherical hexagon occurring as the union of triangles in $\mathcal{D}_{X}$.

\begin{Lemma}(Hexagon Lemma) \\
Let $\alpha,\beta,\alpha',\beta',\gamma$, and $\omega$ denote the internal angles of $C_6$ as shown in Figure \ref{fig:HexagonLemma1}. If $\alpha, \beta, \gamma \in \{2\pi-k\arccos(1/3) \; | \; k=3,4\}$, then $\alpha' + \beta' + \omega \notin \{2\pi - k\arccos(1/3) \; | \; k=1,2,3,4\}$.

Let $\alpha, \beta, \theta, \beta', \gamma'$, and $\omega$ denote the internal angles of $C_6'$ as shown in Figure \ref{fig:HexagonLemma2}. If $\alpha, \beta \in \{2\pi-k\arccos(1/3) \; | \; k=3,4\}$ and $\theta \in \{2\pi-k\arccos(1/3) \; | \; k=1,2,3,4\}$, then $\beta' + \gamma' + \omega \notin \{2\pi - k\arccos(1/3) \; | \; k=1,2,3,4\}$.

Let $\alpha, \beta, \theta, \gamma'$, and $\omega$ denote the internal angles of $C_6''$ as shown in Figure \ref{fig:HexagonLemma3}. If $\alpha, \beta \in \{2\pi-k\arccos(1/3) \; | \; k=3,4\}$ and $\theta \in \{2\pi-k\arccos(1/3) \; | \; k=1,2,3,4\}$, then $\omega + \gamma' \notin \{2\pi-k\arccos(1/3) \; | \; k=1,2,3,4\}$.
\end{Lemma}
\begin{proof}
For the case of $C_6$, let $\alpha, \beta, \gamma \in \{2\pi-k\arccos(1/3) \; | \; k=3,4\}$ and let $a$, $b$, and $c$ denote the side lengths opposite to $\alpha$, $\beta$, and $\gamma$ in the corresponding type I triangles. Then we can compute the side lengths $a$, $b$, and $c$ and the internal angles $\alpha',\beta',\omega$ by the first law of cosines as,
\begin{equation*}
a = \arccos\left(\frac{1+3\cos\alpha}{4}\right),\; b = \arccos\left(\frac{1+3\cos\beta}{4}\right),\; c = \arccos\left(\frac{1+3\cos\gamma}{4}\right)
\end{equation*}

\begin{equation*}
\alpha'=\arccos\left(\frac{1-\cos a}{\sqrt{3}\sin a}\right) \;\;\;\;\;\;
\beta'=\arccos\left(\frac{1-\cos b}{\sqrt{3}\sin b}\right)
\end{equation*}

\begin{equation*}
\omega = \arccos\left(\frac{\cos c-\cos a\cos b}{\sin a\sin b}\right)
\end{equation*}

We now consider our possible cases for varying $\alpha$, $\beta$, and $\gamma$. By symmetry we need to look at only the following cases,
\begin{center}
Table 3. Cases for $C_6$ of the Hexagon Lemma
\begin{tabular}{| c | c | c | c | c | c | c | c | c | c |} 
\hline
$\alpha$ & $\beta$ & $\gamma$ & $a$ & $b$ & $c$ & $\alpha'$ & $\beta'$ & $\omega$ & $\alpha' + \beta' + \omega$ \\ \hline
1.359 & 1.359 & 1.359 & 1.151 & 1.151 & 1.151 & 1.186 & 1.186 & 1.277 & 3.650 \\ \hline
1.359 & 1.359 & 2.590 & 1.151 & 1.151 & 1.970 & 1.186 & 1.186 & 2.298 & 4.672 \\ \hline
1.359 & 2.590 & 2.590 & 1.151 & 1.970 & 1.970 & 1.186 & 0.514 & 1.848 & 3.549 \\ \hline
2.590 & 2.590 & 2.590 & 1.970 & 1.970 & 1.970 & 0.514 & 0.514 & 2.260 & 3.290 \\ \hline
\end{tabular}
\end{center}
Therefore, we have that in each case,
$\alpha' + \beta' + \omega \notin \{2\pi - k\arccos(1/3) \; | \; k=1,2,3,4\}$.

\begin{figure}[H]
\begin{center}
\includegraphics[scale=0.2]{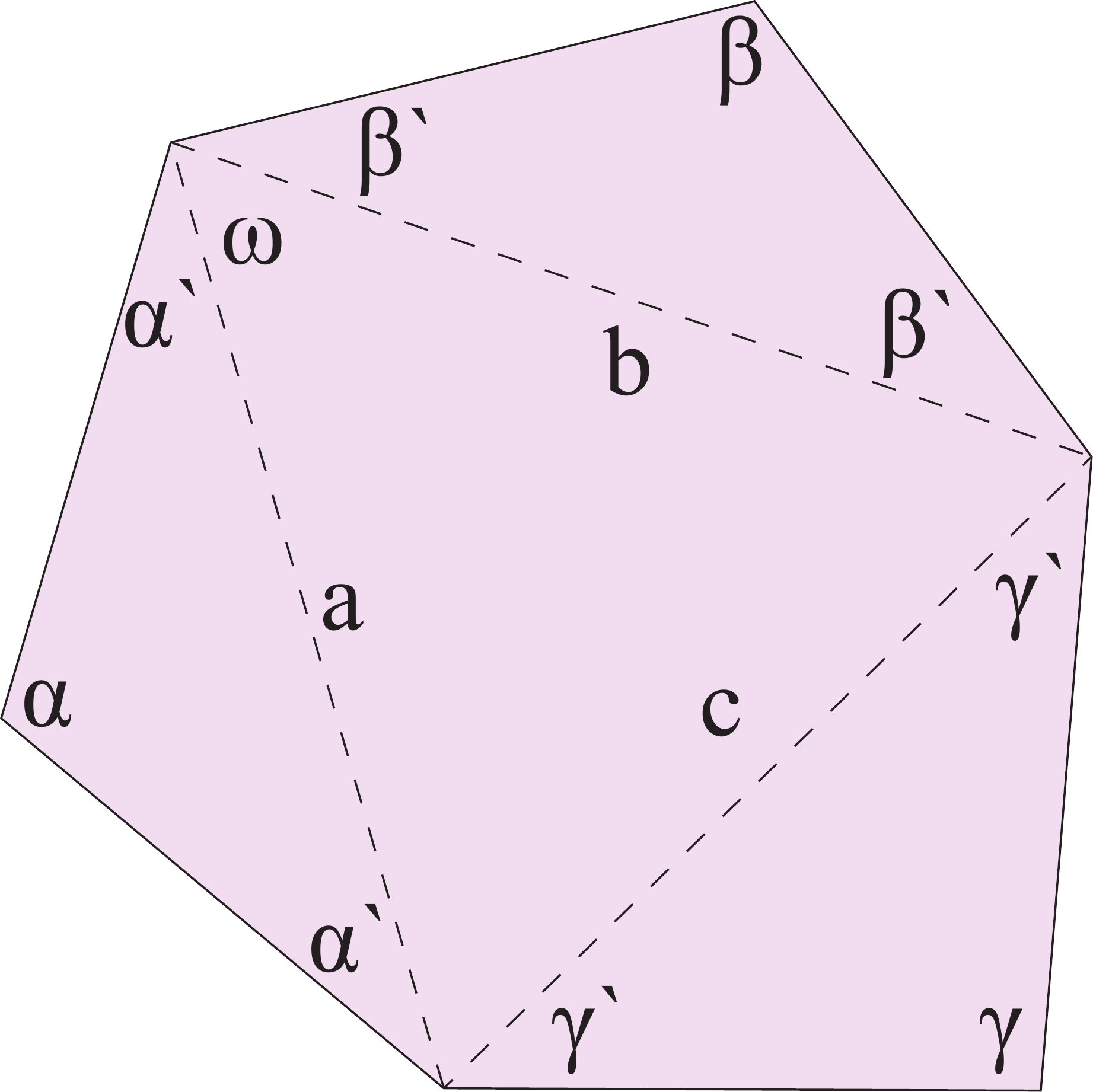}
\caption{A triangulated $C_6$ hexagon with side lengths of $\pi/3$.}
\label{fig:HexagonLemma1}
\end{center}
\end{figure}

For the case of $C_6'$, let $\alpha, \beta \in \{2\pi-k\arccos(1/3) \; | \; k=3,4\}$ and let $\theta \in \{2\pi - k\arccos(1/3) \; | \; k=1,2,3,4\}$. We use the formulas for $a,b,c,\alpha'$, and $\beta'$ mentioned in the case of $C_6$. Then, $\gamma = \theta - \alpha'$ and we compute,

\begin{equation*}
\gamma' = \arccos\left(\frac{\cos a -\cos(\pi/3)\cos c}{\sin(\pi/3)\sin c}\right)
\end{equation*}

\begin{equation*}
\omega = \arccos\left(\frac{\cos(\pi/3)-\cos b \cos c }{\sin b \sin c}\right)
\end{equation*}

We now consider 16 cases by varying $\alpha, \beta$, and $\theta$. For simplicity, we have mentioned every case by disregarding any symmetries present in $C_6'$. Furthermore, observe that in the column of value for $c$, $0.149< \pi/3$ and $0.725<\pi/3$, so these cases are not realizable.
\begin{center}
Table 4. Cases for $C_6'$ of the Hexagon Lemma
\begin{tabular}{| c | c | c | c | c | c | c | c | c | c | c | c |} 
\hline
$\alpha$ & $\beta$ & $\theta$ & $\gamma$ & $a$ & $b$ & $c$ & $\alpha'$ & $\beta'$ & $\gamma'$ & $\omega$ & $\beta' + \gamma' + \omega$ \\ \hline
1.359 & 1.359 & 1.359 & 0.172 & 1.151 & 1.151 & 0.149 & - & - & - & - & - \\ \hline
1.359 & 1.359 & 2.590 & 1.403 & 1.151 & 1.151 & 1.186 & 1.186 & 1.186 & 1.293 & 1.148 & 3.625 \\ \hline
1.359 & 1.359 & 3.821 & 2.634 & 1.151 & 1.151 & 1.988 & 1.186 & 1.186 & 0.690 & 0.648 & 2.525 \\ \hline
1.359 & 1.359 & 5.052 & 3.865 & 1.151 & 1.151 & 1.888 & 1.186 & 1.186 & 0.816 & 0.763 & 2.766 \\ \hline
1.359 & 2.590 & 1.359 & 0.172 & 1.151 & 1.970 & 0.149 & - & - & - & - & - \\ \hline
1.359 & 2.590 & 2.590 & 1.403 & 1.151 & 1.970 & 1.186 & 1.186 & 0.514 & 1.293 & 0.713 & 2.521 \\ \hline
1.359 & 2.590 & 3.821 & 2.634 & 1.151 & 1.970 & 1.988 & 1.186 & 0.514 & 0.690 & 1.152 & 2.357 \\ \hline
1.359 & 2.590 & 5.052 & 3.865 & 1.151 & 1.970 & 1.888 & 1.186 & 0.514 & 0.816 & 1.123 & 2.454 \\ \hline
2.590 & 1.359 & 1.359 & 0.844 & 1.970 & 1.151 & 0.725 & - & - & - & - & - \\ \hline
2.590 & 1.359 & 2.590 & 2.075 & 1.970 & 1.151 & 1.683 & 0.514 & 1.186 & 1.967 & 0.925 & 4.079 \\ \hline
2.590 & 1.359 & 3.821 & 3.306 & 1.970 & 1.151 & 2.082 & 0.514 & 1.186 & 1.762 & 0.497 & 3.447 \\ \hline
2.590 & 1.359 & 5.052 & 4.537 & 1.970 & 1.151 & 1.451 & 0.514 & 1.186 & 2.119 & 1.049 & 4.356 \\ \hline
2.590 & 2.590 & 1.359 & 0.844 & 1.970 & 1.970 & 0.725 & - & - & - & - & - \\ \hline
2.590 & 2.590 & 2.590 & 2.075 & 1.970 & 1.970 & 1.683 & 0.514 & 0.514 & 1.967 & 1.049 & 3.531 \\ \hline
2.590 & 2.590 & 3.821 & 3.306 & 1.970 & 1.970 & 2.082 & 0.514 & 0.514 & 1.762 & 1.175 & 3.452 \\ \hline
2.590 & 2.590 & 5.052 & 4.537 & 1.970 & 1.970 & 1.451 & 0.514 & 0.514 & 2.119 & 0.930 & 3.565 \\ \hline
\end{tabular}
\end{center}
Therefore, we have that in each case,
$\beta' + \gamma' + \omega \notin \{2\pi - k\arccos(1/3) \; | \; k=1,2,3,4\}$.

\begin{figure}[H]
\begin{center}
\includegraphics[scale=0.23]{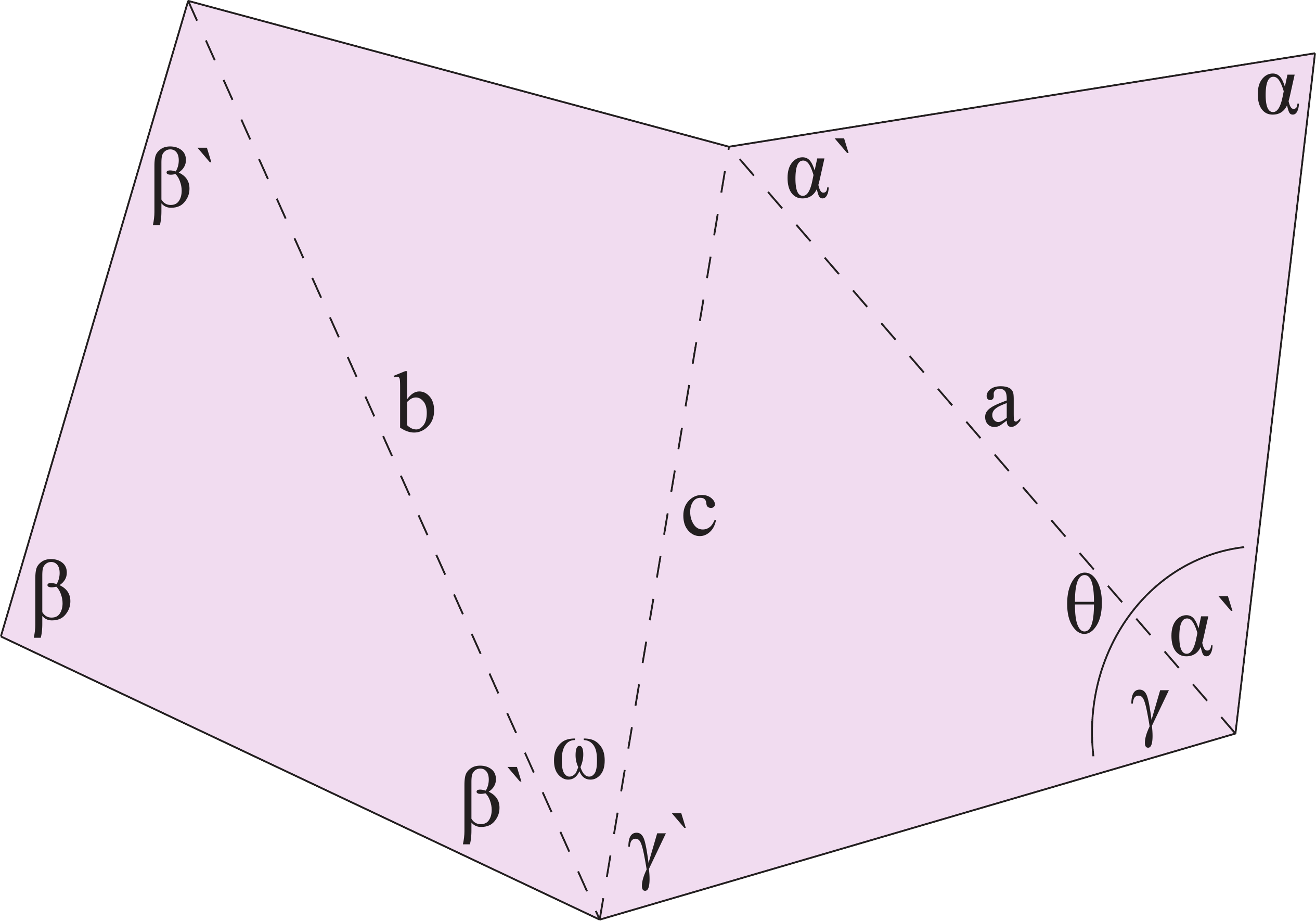}
\caption{A triangulated $C_6'$ hexagon with side lengths of $\pi/3$.}
\label{fig:HexagonLemma2}
\end{center}
\end{figure}

For the case of $C_6''$, let $\alpha, \beta \in \{2\pi-k\arccos(1/3) \; | \; k=3,4\}$ and let $\theta \in \{2\pi - k\arccos(1/3) \; | \; k=1,2,3,4\}$. We use the formulas for $a,b,c,\alpha'$, and $\gamma'$ in the case of $C_6$. Then, $\gamma = \theta - \alpha'$ and we compute,
$$\omega = \arccos\left(\frac{\cos b - \cos(\pi/3)\cos c}{\sin(\pi/3)\sin c}\right)$$

We now consider 16 cases by varying $\alpha, \beta$, and $\theta$. For simplicity, we have mentioned every case by disregarding any symmetries present in $C_6''$. Furthermore, observe that in the column of value for $c$, $0.149< \pi/3$ and $0.725<\pi/3$, so these cases are not realizable.

\bigskip

\begin{center}
Table 5. Cases for $C_6''$ of the Hexagon Lemma
\begin{tabular}{| c | c | c | c | c | c | c | c | c | c | c | c |} 
\hline
$\alpha$ & $\beta$ & $\theta$ & $\alpha'$ & $\gamma$ & $a$ & $b$ & $c$ & $\gamma'$ & $\omega$ & $\gamma' + \omega$ \\ \hline
1.359 & 1.359 & 1.359 & 1.186 & 0.172 & 1.151 & 1.151 & 0.149 & - & - & - \\ \hline
1.359 & 1.359 & 2.590 & 1.186 & 1.403 & 1.151 & 1.151 & 1.186 & 1.170 & 1.293 & 2.464 \\ \hline
1.359 & 1.359 & 3.821 & 1.186 & 2.634 & 1.151 & 1.151 & 1.988 & 0.478 & 0.690 & 1.168 \\ \hline
1.359 & 1.359 & 5.052 & 1.186 & 3.865 & 1.151 & 1.151 & 1.888 & 0.648 & 0.816 & 1.464 \\ \hline
1.359 & 2.590 & 1.359 & 1.186 & 0.172 & 1.151 & 1.970 & 0.149 & - & - & - \\ \hline
1.359 & 2.590 & 2.590 & 1.186 & 1.403 & 1.151 & 1.970 & 1.186 & 1.170 & 2.371 & 3.542 \\ \hline
1.359 & 2.590 & 3.821 & 1.186 & 2.634 & 1.151 & 1.970 & 1.988 & 0.478 & 1.808 & 2.286 \\ \hline
1.359 & 2.590 & 5.052 & 1.186 & 3.865 & 1.151 & 1.970 & 1.888 & 0.648 & 1.857 & 2.505 \\ \hline
2.590 & 1.359 & 1.359 & 0.514 & 0.844 & 1.970 & 1.151 & 0.725 & - & - & - \\ \hline
2.590 & 1.359 & 2.590 & 0.514 & 2.075 & 1.970 & 1.151 & 1.683 & 0.867 & 1.001 & 1.869 \\ \hline
2.590 & 1.359 & 3.821 & 0.514 & 3.306 & 1.970 & 1.151 & 2.082 & 0.163 & 0.527 & 0.691 \\ \hline
2.590 & 1.359 & 5.052 & 0.514 & 4.537 & 1.970 & 1.151 & 1.451 & 1.033 & 1.154 & 2.187 \\ \hline
2.590 & 2.590 & 1.359 & 0.514 & 0.844 & 1.970 & 1.970 & 0.725 & - & - & - \\ \hline
2.590 & 2.590 & 2.590 & 0.514 & 2.075 & 1.970 & 1.970 & 1.683 & 0.867 & 1.967 & 2.835 \\ \hline
2.590 & 2.590 & 3.821 & 0.514 & 3.306 & 1.970 & 1.970 & 2.082 & 0.163 & 1.762 & 1.926 \\ \hline
2.590 & 2.590 & 5.052 & 0.514 & 4.537 & 1.970 & 1.970 & 1.451 & 1.033 & 2.119 & 3.152 \\ \hline
\end{tabular}
\end{center}
Therefore, we have that in each case,
$\gamma' + \omega \notin \{2\pi-k\arccos(1/3) \; | \; k=1,2,3,4\}$. 
\begin{figure}[H]
\begin{center}
\includegraphics[scale=0.2]{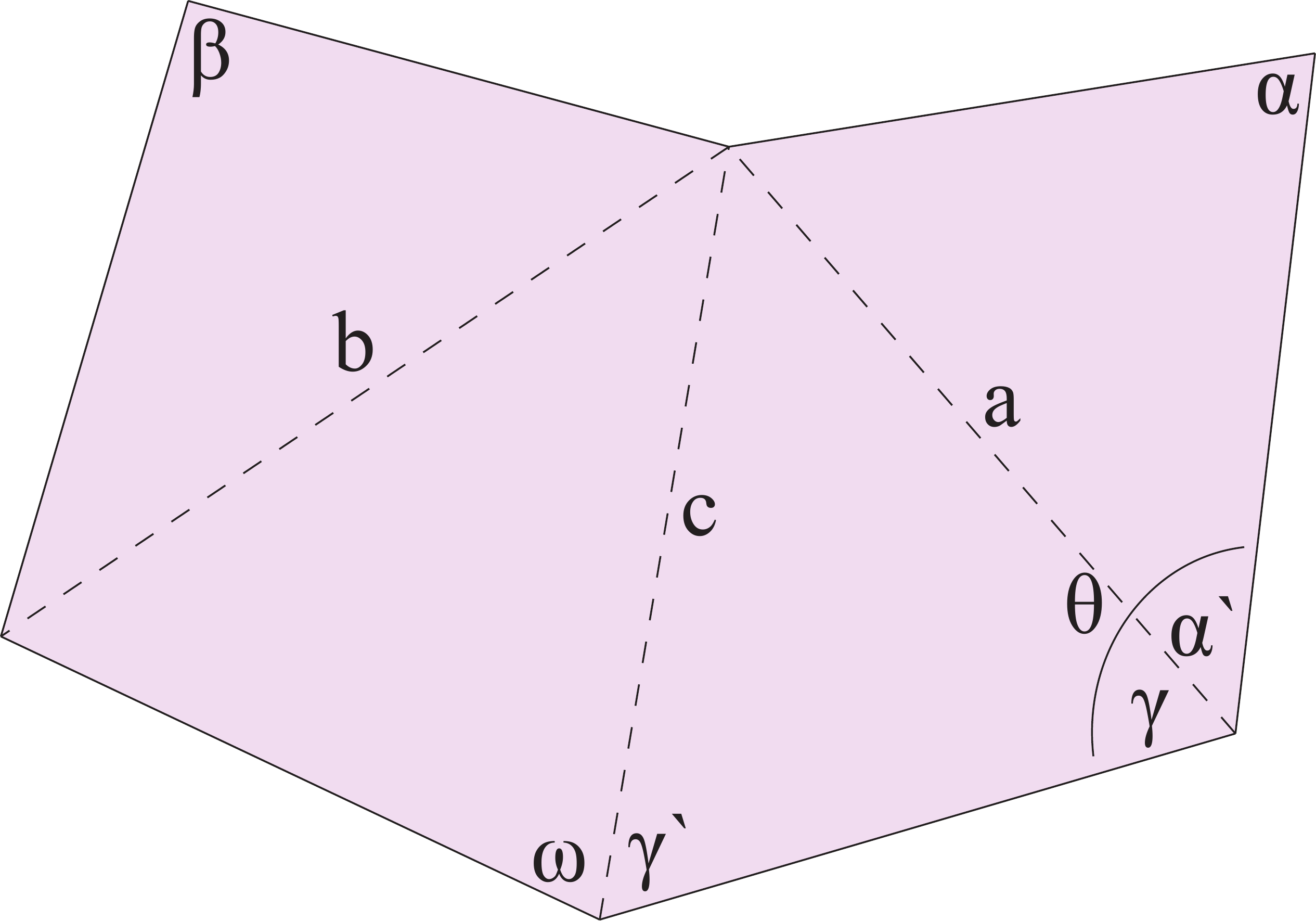}
\caption{A triangulated $C_6''$ hexagon with side lengths of $\pi/3$.}
\label{fig:HexagonLemma3}
\end{center}
\end{figure}
\end{proof}

\section{Proof of Theorem~\ref{spherical-main-theorem}}

We first mention a rather straightforward but important fact which will be used throughout the following two sub-sections.

\begin{Remark}\label{trivial}
Let $X$ be a point set on $\mathbb{S}^2$ with minimum spherical distance $\pi/3$. Then
a point in $X$ cannot be entirely surrounded by regular Delaunay triangles of side length $\pi/3$ since $\frac{2\pi} {\arccos(1/3)} \notin \mathbb{N}$.
\end{Remark}

\subsection{An upper bound on touching triplets on $\mathbb{S}^2$}

Our main goal is to show that the number of touching triplets in an arbitrary packing of spherical caps of angular radius $\pi/6$ on $\mathbb{S}^2$ is at most 11.

Let $X=\{\mathbf{x}_{1},...,\mathbf{x}_{N}\}$ be a point set on $\mathbb{S}^2$ with $N \leq 12$ and minimum spherical distance $\pi/3$ between any two points in $X$. Taking the Delaunay triangulation $\mathcal{D}_{X}$ of $X$ on $\mathbb{S}^2$, we let $f$ be the number of faces, $e$ be the number of edges, and $N$ be the number of vertices. Since $\mathcal{D}_{X}$ is a triangulation of $X$ on $\mathbb{S}^2$, every face is a triangle, and so $3f=2e$. Using Euler's formula $N-e+f=2$, it is straightforward to see that $f=2N-4$. 

Assume to the contrary that there exist at least $12$ touching triplets in a packing of $N$  spherical caps of angular radius $\pi/6$ on $\mathbb{S}^2$ with the center points forming $X$ and with $N\le 12$.

Assume that $N=12$. Then $f=2N-4=2(12)-4=20$, so there are 20 triangles in $\mathcal{D}_{X}$. We then have that $\mathcal{D}_{X}$ consists of at least $12$ regular triangles (of side length $\frac{\pi}{3}$) and at most $8$ irregular triangles. (Actually, for the purpose of the proof below, any regular triangle of $\mathcal{D}_{X}$ having side length $>\frac{\pi}{3}$ is listed among the irregular ones.) In what follows we assume that the number of irregular triangles is $8$. Namely, if we have fewer than $8$, then the analysis of the cases is a simpler version of what we do here, so we leave it to the reader. Now, we can determine the possible cases for the union of the irregular triangles as $8=4+4=3+5=2+6=2+3+3=2+2+4=2+2+2+2$, where each number represents a (dashed)side-to-(dashed)side union of that many irregular triangles. We then have the following cases: $(1)_{12}$ One decagon $(2)_{12}$ Two hexagons, $(3)_{12}$ One pentagon and one heptagon, $(4)_{12}$ One quadrilateral and one octagon, $(5)_{12}$ One quadrilateral and two pentagons, $(6)_{12}$ Two quadrilaterals and one hexagon, and $(7)_{12}$ Four quadrilaterals.

For $(1)_{12}$, we have that there are $10$ vertices of the decagon, so there exists two vertices in $X$ which are entirely surrounded by regular triangles. This contradicts Remark~\ref{trivial}.

For $(2)_{12}$, we have that there are $12$ vertices cumulatively from the two hexagons, so since we have assumed $N=12$, they are pairwise disjoint (see Remark~\ref{trivial}). Applying the Hexagon Lemma to either hexagon, we have a contradiction since both hexagons are assumed to be entirely surrounded by regular triangles.

For $(3)_{12}$, we have that there are $12$ vertices cumulatively from the pentagon and the heptagon, so since we have assumed $N=12$, they are pairwise disjoint (see Remark~\ref{trivial}). Applying the Pentagon Lemma to the pentagon, we have a contradiction since the pentagon is assumed to be entirely surrounded by regular triangles.

For $(4)_{12}$, we have that there are $12$ vertices cumulatively from the quadrilateral and the octagon, so since we have assumed $N=12$, they are pairwise disjoint (see Remark~\ref{trivial}). Applying the Quadrilateral Lemma to the quadrilateral, we have a contradiction since the quadrilateral is assumed to be entirely surrounded by regular triangles.

For $(5)_{12}$, we have that there are $14$ vertices cumulatively from the quadrilateral and the two pentagons. If two of the polygons share an edge, then there is either a quadrilateral or a pentagon which is entirely surrounded by regular triangles, leading to a contradiction by the Quadrilateral Lemma or Pentagon Lemma. If the three polygons share precisely one vertex in common, then the Quadrilateral Lemma leads to a contradiction. So, the only possible configuration left is to have one polygon $C$ to share one vertex with another polygon $C'$ with the last polygon $C''$ sharing a vertex with $C'$ (and with the two shared vertices being distinct). Now, if $C$ or $C''$ is a quadrilateral or if $C'$ is a quadrilateral with the two shared vertices being adjacent in $C'$, then the Quadrilateral Lemma leads to a contradiction.

\begin{figure}[H]
\begin{center}
\includegraphics[scale=0.16]{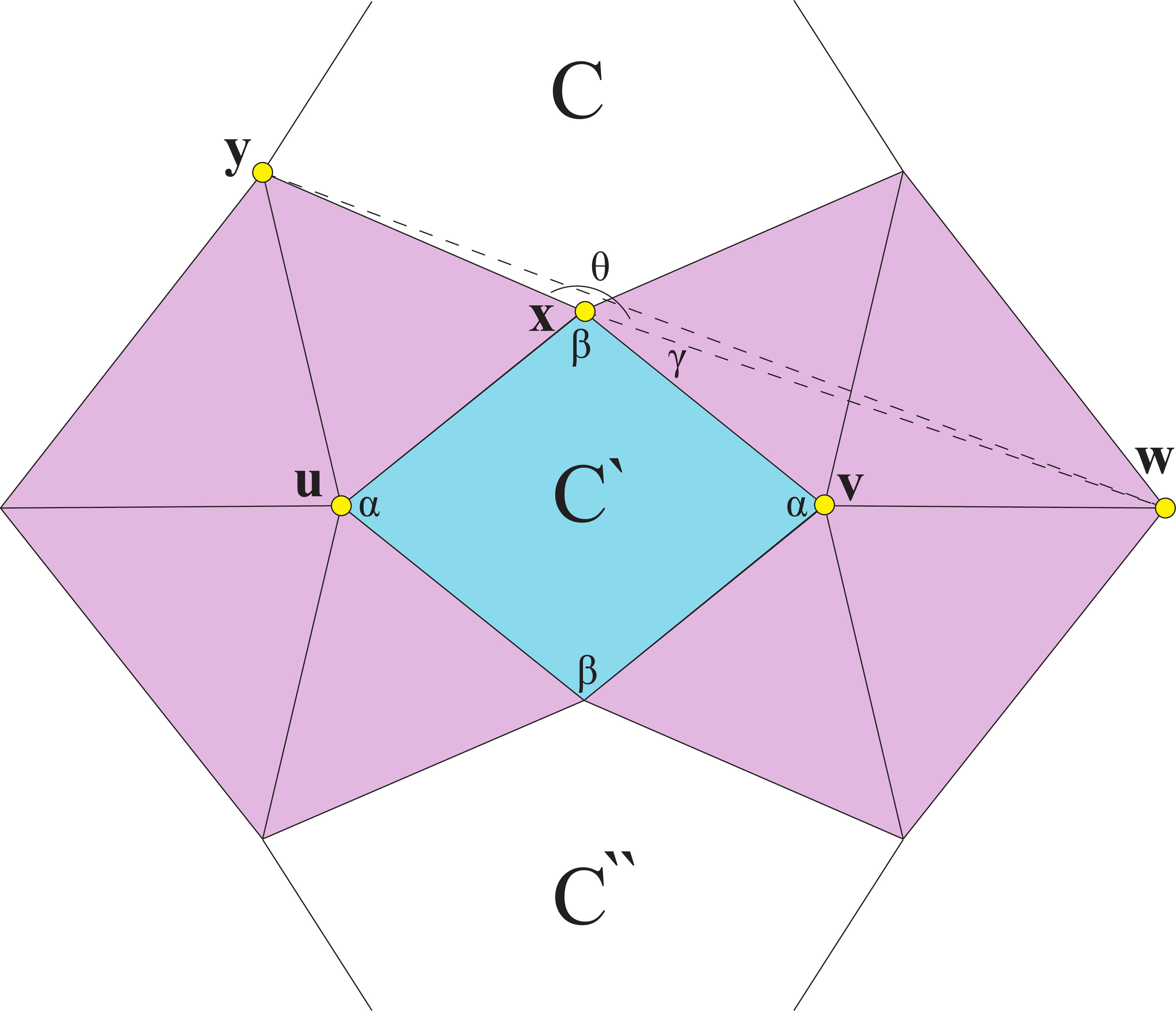}
\caption{The configuration described in case $(5)_{12}$.}
\label{fig:PaperCase512}
\end{center}
\end{figure}

Thus, we are left with the following configuration: $C$ and $C''$ are pentagons (of side length $\frac{\pi}{3}$) and $C'$ is a quadrilateral (of side length $\frac{\pi}{3}$) with the two shared vertices $\mathbf u$ and $\mathbf v$ being opposite in $C'$.
Referring to Figure \ref{fig:PaperCase512} for this configuration, we let $a$ be the distance between the vertices $\mathbf w$ and $\mathbf x$ and $b$ be the distance between the vertices $\mathbf w$ and $\mathbf y$. Then by the proof of the Quadrilateral Lemma, we have that $\alpha = 1.359$ and $\beta=2.373$. Given that the triangles at $\mathbf u$ and $\mathbf v$ in Figure \ref{fig:PaperCase512} are assumed to be regular, we can compute $b$ as follows. Using the spherical law of cosines, we compute the following where $\theta = 2\pi-(\arccos\frac{1}{3}+\beta+\gamma)$ as seen in Figure \ref{fig:PaperCase512},

\begin{equation*}
a = \arccos\left(\frac{1+3\cos(2\arccos\frac{1}{3})}{4}\right) = 1.91...
\end{equation*}
\begin{equation*}
\gamma = \arccos\left(\frac{\cos\frac{\pi}{3} - \cos\frac{\pi}{3}\cos a}{\sin\frac{\pi}{3}\sin a}\right) = 0.615...
\end{equation*}
\begin{equation*}
b =\arccos\left(\frac{\cos a}{2}+\frac{\sqrt{3}\sin a}{2}\cos \theta\right) = 2.15...
\end{equation*}

Therefore, since $b > 2\pi/3 = 2.09...$, we have that this configuration cannot exist as $\mathbf w$ must be identified, without loss of generality, with a vertex of the pentagon $C$ contradicting 
the rather obvious fact that the spherical diameter of any pentagon with side lengths equal to $\pi/3$ is at most $2\pi/3$.

For $(6)_{12}$, we have that there are 14 vertices cumulatively from the two quadrilaterals and the hexagon. If two of the polygons share an edge, then there is either a quadrilateral or a hexagon which is entirely surrounded by regular triangles, leading to a contradiction by the Quadrilateral Lemma or Hexagon Lemma. If the three polygons share precisely one vertex in common, then the Quadrilateral Lemma leads to a contradiction. So, the only possible configuration left is to have one polygon $C$ to share one vertex with another polygon $C'$ with the last polygon $C''$ sharing a vertex with $C'$ (and with the two shared vertices being distinct). Here either $C$ or $C''$ must be a quadrilateral and so by applying the Quadrilateral Lemma we are led to a contradiction.

For $(7)_{12}$, we show that the total spherical area of the $12$ regular triangles (of side length $\frac{\pi}{3}$) and of the four quadrilaterals (of side length $\frac{\pi}{3}$) is $<4\pi$, a contradiction. Indeed, the discrete isoperimetric inequality of spherical polygons (see for example \cite{CsLaNa06}) implies that the spherical area of a quadrilateral of side length $\frac{\pi}{3}$ in $\mathbb{S}^2$ is maximal when it is regular. Thus, the total spherical area of the $12$ regular triangles and of the $4$ quadrilaterals in question is at most
$$
12\left(6\arcsin\frac{1}{\sqrt{3}}-\pi\right)+4\left(8\arctan\sqrt{2}-2\pi\right)=12.052\ldots<4\pi=12.566\ldots\
$$
which is a contradiction.

Since the cases $(1)_{12}$, $(2)_{12}$, $(3)_{12}$, $(4)_{12}$, $(5)_{12}$, $(6)_{12}$, and $(7)_{12}$ all lead to a contradiction, we have that the number of touching triplets in an arbitrary packing of $12$ spherical caps of angular radius $\pi/6$ on $\mathbb{S}^2$ is at most $11$.

Assume that $N=11$, then $f=2N-4=2(11)-4=18$, so there are $18$ triangles in $\mathcal{D}_{X}$. Hence, by our indirect assumption, $\mathcal{D}_{X}$ consists of at least $12$ regular triangles and at most $6$ irregular triangles. (Just as above, any regular triangle of $\mathcal{D}_{X}$ having side length $>\frac{\pi}{3}$ is listed among the irregular ones.) In what follows we assume that the number of irregular triangles is $6$ and leave the analysis of the simpler case of less than $6$ irregular triangles to the reader. So, we can determine the possible cases for the union of the irregular triangles as $6=3+3=2+4=2+2+2$, where each number represents a (dashed)side-to-(dashed)side union of that many irregular triangles. We then have the following cases: $(1)_{11}$ One octagon, $(2)_{11}$ Two pentagons, $(3)_{11}$ One quadrilateral and one hexagon, and $(4)_{11}$ Three quadrilaterals.

For $(1)_{11}$, we have that there are $8$ vertices of the octagon, so since we have assumed $N=11$, there are three vertices in $X$ entirely surrounded by regular triangles which is a contradiction (see Remark~\ref{trivial}).

For $(2)_{11}$ and $(3)_{11}$, we have that there are $10$ vertices cumulatively from the two pentagons or from the quadrilateral and the hexagon, so since we have assumed $N=11$, this is a contradiction since there is one vertex in $X$ entirely surrounded by regular triangles (see Remark~\ref{trivial}).

For $(4)_{11}$, we have that there are $12$ vertices cumulatively from the three quadrilaterals, so since we have assumed $N=11$, exactly one vertex is shared by the polygons. By applying the Quadrilateral Lemma to the quadrilateral which does not share a vertex with any other quadrilateral, we arrive at a contradiction. 

Since the cases $(1)_{11}$, $(2)_{11}$, $(3)_{11}$, and $(4)_{11}$ all lead to a contradiction, we have that the number of touching triplets in an arbitrary packing of $11$ spherical caps with angular radius $\pi/6$ on $\mathbb{S}^2$ is at most $11$.

Assume that $N=10$. Then $f=2N-4=2(10)-4=16$ and so, there are $16$ triangles in $\mathcal{D}_{X}$. Hence, by our indirect assumption, $\mathcal{D}_{X}$ consists of at least $12$ regular triangles and at most $4$ irregular triangles. Based on $4=2+2$ we can determine the possible cases for the union of the irregular triangles (by leaving the study of the simplier case of less than $4$ irregular triangles to the reader): $(1)_{10}$ One hexagon, and $(2)_{10}$ Two quadrilaterals.

For $(1)_{10}$ as well as $(2)_{10}$, we have that there is a vertex in $X$ which is  entirely surrounded by regular triangles contradicting to Remark~\ref{trivial}.

Since Case $(1)_{10}$ and $(2)_{10}$ both lead to a contradiction, we have that the number of touching triplets in an arbitrary packing of $10$ spherical caps with angular radius $\pi/6$ on $\mathbb{S}^2$ is at most $11$.

Finally, assume that $N \leq 9$. Hence, $f\leq 2(9) - 4 = 14$, and by our indirect assumption there are at least $12$ regular triangles (of side length $\frac{\pi}{3}$) in $\mathcal{D}_{X}$ and at most two irregular ones whose union then must be a quadrilateral (of side length $\frac{\pi}{3}$). This case is clearly impossible by the spherical area estimate of $(7)_{12}$.

This finishes our indirect proof on the number of touching triplets in Theorem~\ref{spherical-main-theorem}.

\subsection{An upper bound on touching pairs on $\mathbb{S}^2$}

Our goal is to show that the number of touching pairs in an arbitrary packing of spherical caps of angular radius $\pi/6$ on $\mathbb{S}^2$ is at most $25$. The proof presented here is indirect and it is based on the previous section. The details are as follows.

Let $X=\{\mathbf{x}_{1},...,\mathbf{x}_{N}\}$ be a point set on $\mathbb{S}^2$ with $N \leq 12$ and minimum spherical distance $\pi/3$ between any two points in $X$. Taking the Delaunay triangulation $\mathcal{D}_{X}$ of $X$ on $\mathbb{S}^2$, we let $f$ be the number of faces, $e$ be the number of edges, and $N$ be the number of vertices. Since $\mathcal{D}_{X}$ is a triangulation of $X$ on $\mathbb{S}^2$, every face is a triangle, and so $e=\frac{3}{2}f$. Moreover, just as in the previous section, Euler's formula implies that $f=2N-4$. 

Now, assume to the contrary that there exist at least $26$ touching pairs in a packing of $N$  spherical caps of angular radius $\pi/6$ on $\mathbb{S}^2$ with the center points forming $X$ and with $N\le 12$. Then $e=3(N-2)\ge 26$ implies that either $N=12$ or $N=11$.

If $N=12$, then $f=20$ and $e=30$. Hence, the indirect assumption implies that the number of edges of length $>\frac{\pi}{3}$ of $\mathcal{D}_{X}$ is at most $4$ and so, there are at most $8$ irregular triangles in $\mathcal{D}_{X}$. (Here an irregular triangle of $\mathcal{D}_{X}$ means a triangle of $\mathbb{S}^2$ different from the one having side lengths equal to $\frac{\pi}{3}$.) Thus, one can repeat the proof of the previous section under the case $N=12$ leading to a contradiction. 

Finally, if $N=11$, then $f=18$ and $e=27$. Hence, the indirect assumption implies that the number of edges of length $>\frac{\pi}{3}$ of $\mathcal{D}_{X}$ is at most $1$ and so, there are at most $2$ irregular triangles in $\mathcal{D}_{X}$. (Here again an irregular triangle of $\mathcal{D}_{X}$ means a triangle of $\mathbb{S}^2$ different from the one having side lengths equal to $\frac{\pi}{3}$.) Thus, one can repeat the proof of the previous section under the case $N=11$ leading to a contradiction.

This finishes our indirect proof on the number of touching pairs in Theorem~\ref{spherical-main-theorem}.

\section{Explicit Constructions}

\subsection{The Octahedral Construction}

We now consider an explicit construction of a unit sphere packing placed over the face-centered cubic lattice for which we obtain a high number of touching triplets and quadruples of unit balls in $\mathbb{E}^3$. For any positive integer $k \geq 2$, place $n(k)=\frac{2k^{3}+k}{3}$ lattice points of the face-centered cubic lattice such that their convex hull is a regular octahedron $\mathbf{K} \subset \mathbb{E}^3$ of edge length $2(k-1)$ having exactly $k$ lattice points along each of its edges (see Figure \ref{fig:OctahedralConstruction} for $k=4$).

\begin{figure}[H]
\begin{center}
\includegraphics[scale=0.6]{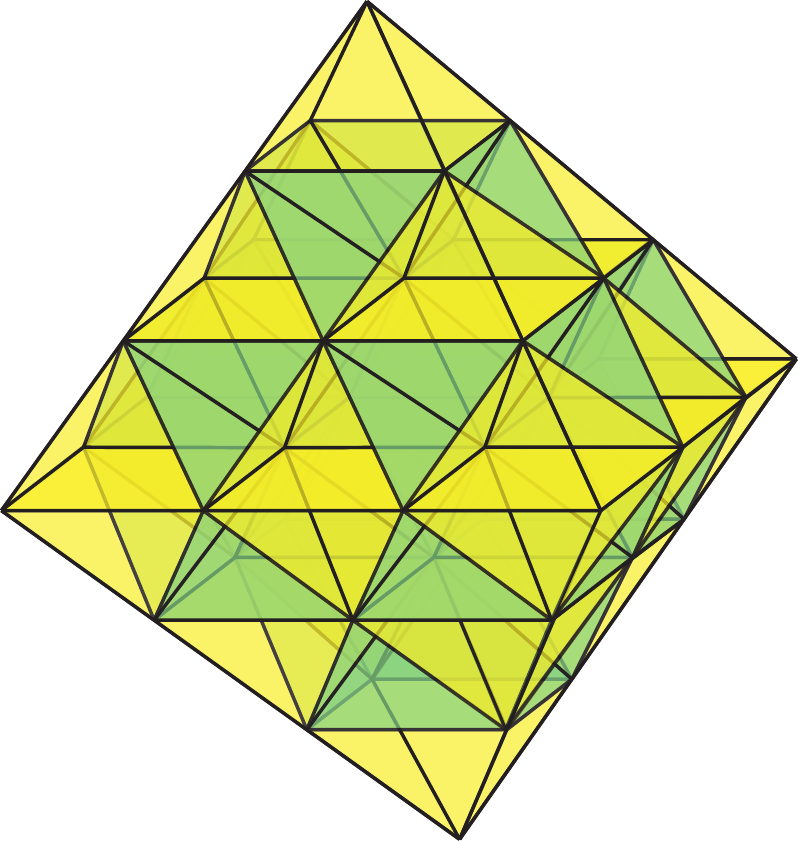}
\caption{The Octahedral Construction for $k=4$.}
\label{fig:OctahedralConstruction}
\end{center}
\end{figure}

It is not hard to see via layer by layer cross-sections that the number of regular tetrahedra of edge length 2 spanned by the lattice points in $\mathbf{K}$ is equal to  $$2\sum\limits_{i=2}^{k-1} 2i(i-1) = \frac{4(k-2)(k-1)k}{3}\ .$$  
This means that we can pack $n(k)=\frac{2k^{3}+k}{3}$ unit balls forming an octahedral shape in $\mathbb{E}^3$ such that there are exactly $N_{4}(k)=\frac{4(k-2)(k-1)k}{3}$ touching quadruples. We note that $\frac{3}{2}n(k) > k^{3}$ implies $\left(\frac{3}{2}\right)^{2/3}n^{2/3}(k) > k^2$, and observe that $n(k) < k^{3}$ implies $n^{1/3}(k) < k$. With these bounds, we can bound $N_{4}(k)$ in terms of $n$.
$$N_{4}(k)=\frac{4(k-2)(k-1)k}{3}=2\Big(\underbrace{\frac{2}{3}k^{3}+\frac{1}{3}k}_{\text{n(k)}}\Big) - 4k^{2} + 2k > 2n(k) - 4\left(\frac{3}{2}\right)^{2/3}n^{2/3}(k) + 2n^{1/3}(k)\ ,$$
finishing the proof of $(iii)$ in Theorem~\ref{main-theorem} on touching quadruples.


We have seen that the number of regular tetrahedra of edge length 2 in the above mentioned Octahedral Construction is equal to $N_{4}(k)$; we now consider the number of regular triangles of side length $2$ spanned by the lattice points in $\mathbf{K}$ and label it by $N_{3}(k)$. We note that the volume of an octahedron of edge length $2$ is $\frac{8\sqrt{2}}{3}$ moreover, the volume of a tetrahedron of edge length $2$ is $\frac{\sqrt{8}}{3}$. Furthermore, the volume of $\mathbf{K}$ is equal to $\frac{8\sqrt{2}}{3}(k-1)^{3}$. Thus, the number of octahedra of edge length 2 in the Octahedral Construction is equal to,
$$
\frac{\frac{8\sqrt{2}}{3}(k-1)^{3}-\frac{4(k-2)(k-1)k}{3}\frac{\sqrt{8}}{3}}{\frac{8\sqrt{2}}{3}}=(k-1)^{3}-\frac{(k-2)(k-1)k}{3}
$$
We also note that the number of regular triangles of side length 2 of the regular triangle of side length $2(k-1)$, which is the face of $\mathbf{K}$, is equal to $(k-1)^{2}$. Hence, the number of regular triangle faces of side length 2 (i.e., touching triplets) in the Octahedral Construction is,

$$N_{3}(k) = \frac{1}{2} \left(4N_{4}(k) + 8((k-1)^{3}-\frac{(k-2)(k-1)k}{3}(k))+8(k-1)^{2})\right)= \frac{4}{3}(k-1)k(4k-5)\ .$$
Using the bounds introduced before the preceding remark, we can now bound $N_{3}(k)$ in terms of $n$.
$$N_{3}(k) = \frac{4}{3}(k-1)k(4k-5) = 8\Big(\underbrace{\frac{2}{3}k^{3}+\frac{1}{3}k}_{\text{n(k)}}\Big) -12k^{2} + 4k > 8n(k) - 12\left(\frac{3}{2}\right)^{2/3}n^{2/3}(k)+4n^{1/3}(k)\ ,$$
finishing the proof of $(iii)$ in Theorem~\ref{main-theorem} on touching triplets.


\subsection{Tightness of the estimates in Problem~\ref{main-conjecture}}
We need to construct the corresponding polar coordinates for $12$ points on $\mathbb{S}^2$ with minimum spherical distance $\pi/3$ spanning $10$ regular triangles of side length $\pi/3$. By taking the points with the following polar coordinates, we can construct such a point set $P$ as seen in Figure \ref{fig:10tetrahedra}:

\begin{center}
Table 6. Polar coordinates of the points in $P$
\begin{table}[h!]
\begin{center}
\begin{tabular}{|c|c|}
\hline
 & Polar Coordinates \\ \hline
$v_{1}$ & $\Big(1,0,0\Big)$ \\ \hline
$v_{2}$ & $\Big(1,\pi/3,0\Big)$ \\ \hline
$v_{3}$ & $\Big(1,\pi/3,\arccos(1/3)\Big)$ \\ \hline
$v_{4}$ & $\Big(1,\pi/3,2\arccos(1/3)\Big)$ \\ \hline
$v_{5}$ & $\Big(1,\pi/3,3\arccos(1/3)\Big)$ \\ \hline
$v_{6}$ & $\Big(1,\pi/3,4\arccos(1/3)\Big)$ \\ \hline
$v_{7}$ & $\Big(1,\arccos(-7/18),-\arctan(2\sqrt{2}/5)\Big)$ \\ \hline
$v_{8}$ & $\Big(1,2\arctan(\sqrt{2}),\arccos(1/3)/2\Big)$ \\ \hline
$v_{9}$ & $\Big(1,\arccos(-7/18),\pi-\arctan(34\sqrt{2}/19)\Big)$ \\ \hline
$v_{10}$ & $\Big(1,2\arctan(\sqrt{2}),5\arccos(1/3)/2\Big)$ \\ \hline
$v_{11}$ & $\Big(1,2\arctan(\sqrt{2}),7\arccos(1/3)/2\Big)$ \\ \hline
$v_{12}$ & $\Big(1,\arccos(-53/54),\arctan(4\sqrt{2}/17\Big)$ \\ \hline
\end{tabular}
\end{center}
\end{table}
\end{center}

\begin{figure}[H]
\begin{center}
\includegraphics[scale=0.6]{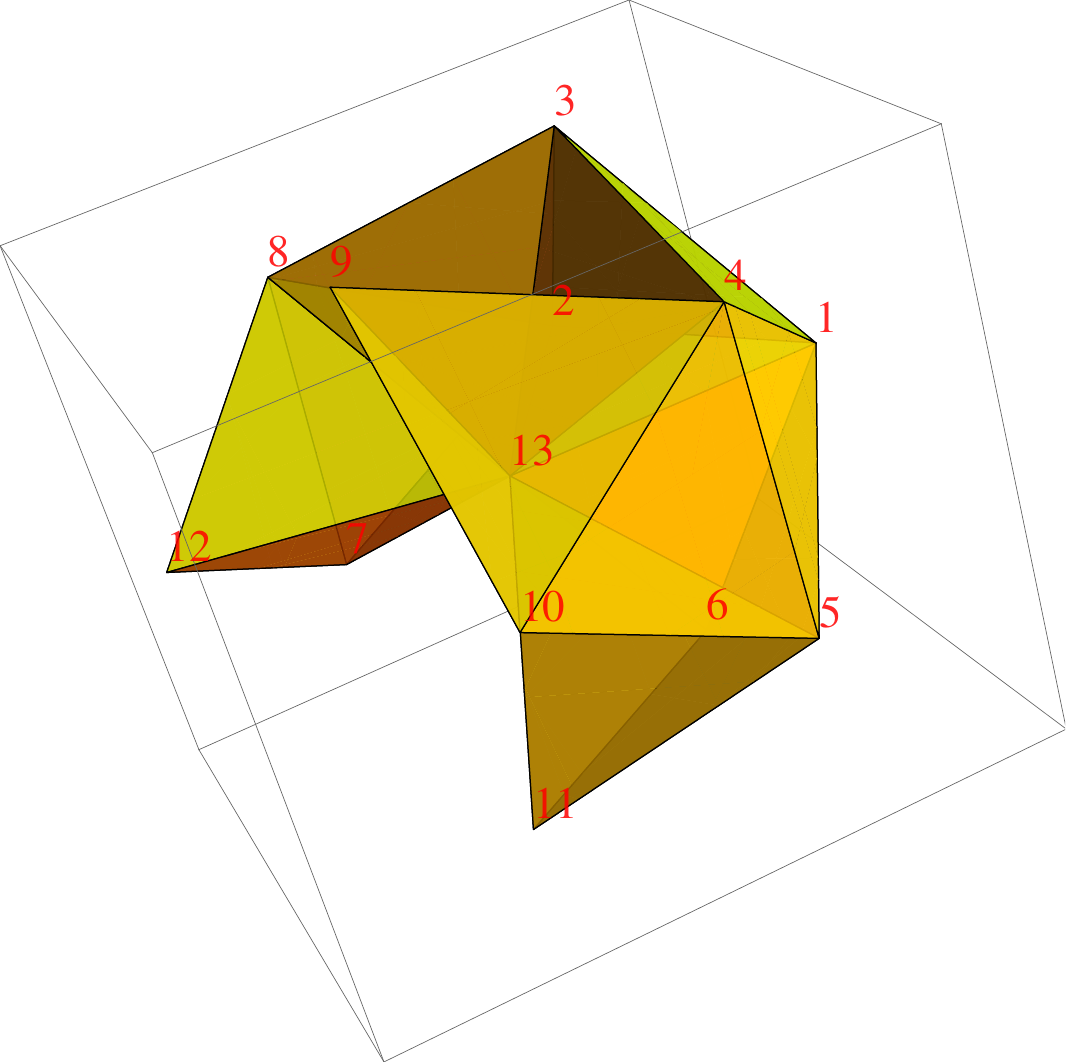}
\caption{An explicit construction of 10 tetrahedra sharing a vertex.}
\label{fig:10tetrahedra}
\end{center}
\end{figure}

Last but not least, we need to generate $12$ points on $\mathbb{S}^2$ with minimum spherical distance $\pi/3$ having $24$ spherical line segments of length $\pi/3$ spanned by the $12$ points. If we place our points at the vertices of a cubeoctahedron with diameter $2$, then we satisfy these conditions.

\vspace{1cm}

\medskip

\noindent
K\'aroly Bezdek
\newline
Department of Mathematics and Statistics, University of Calgary, Canada,
\newline
Department of Mathematics, University of Pannonia, Veszpr\'em, Hungary,
\newline
{\sf E-mail: bezdek@math.ucalgary.ca}
\newline
and
\newline
Samuel Reid
\newline
Department of Mathematics and Statistics, University of Calgary, Canada,
\newline
{\sf E-mail: smrei@ucalgary.ca}

\end{document}